\documentclass[a4paper,10pt,reqno]{amsart}

\usepackage{amsmath, amsthm, amssymb, amsfonts, amscd, mathrsfs, color, relsize}

\newtheorem{theo}{Theorem}[section]
\newtheorem{defi}[theo]{Definition}
\newtheorem{lem}[theo]{Lemma}
\newtheorem{prop}[theo]{Proposition}
\newtheorem{cor}[theo]{Corollary}
\newtheorem{rem}[theo]{Remark}

\newtheorem*{TheoA}{Theorem A}
\newtheorem*{TheoB}{Theorem B}
\newtheorem*{TheoC}{Theorem C}

\numberwithin{equation}{section}

\newcommand{\R}{\mathbb{R}} 
\newcommand{\C}{\mathbb{C}} 
\newcommand{\Z}{\mathbb{Z}} 
\newcommand{\N}{\mathbb{N}} 

\newcommand{\RR}{\mathfrak{S}}         
\newcommand{\ZZ}{\Lambda}              
\newcommand{\ZZdual}{\ZZ^\bot}         
\newcommand{\id}{\mathrm{e}}           
\newcommand{\card}{{|\textnormal{G}|}} 

\newcommand{\T}{\mathcal{T}}           
\newcommand{\TG}{\mathfrak{T}}         
\newcommand{\G}{\mathscr{G}}           
\newcommand{\TT}{{K}}                  
\newcommand{\J}{\mathcal{J}}           
\newcommand{\K}{\mathcal{K}}           
\newcommand{\pj}{\mathcal{P}}          
\newcommand{\A}{\Phi}                  
\newcommand{\V}{\mathcal{V}}           
\newcommand{\W}{\mathcal{W}}           
\newcommand{\AG}{{\A^{\textsc{g}}}}    
\newcommand{\len}{\mathcal{L}}         

\newcommand{\B}{\mathcal{B}}   
\newcommand{\U}{\mathcal{U}}   
\newcommand{\F}{\mathcal{F}}   
\newcommand{\Proj}{\mathbb{P}} 
\newcommand{\Hil}{\mathcal{H}} 

\newcommand{\vsp}{\textnormal{span}}
\newcommand{\ol}[1]{\overline{#1}}

\newcommand{\wt}[1]{\widetilde{#1}}
\newcommand{\wh}[1]{\widehat{#1}}
\newcommand{\one}[1]{\mathlarger{\chi}_{\raisebox{-.5ex}{$\mathsmaller{#1}$}}}

\newcommand{\tr}[1]{T_{#1}}
\newcommand{\rot}[1]{R_{#1}}

\newcommand{\data}{\mathscr{F}}        
\newcommand{\DG}{{\data^{\textsc{g}}}} 
\newcommand{\E}{\mathscr{E}}           

\title{Approximation by group invariant subspaces}
\author[D. Barbieri]{Davide Barbieri}\email{davide.barbieri@uam.es}
\author[C. Cabrelli]{Carlos Cabrelli}\email{cabrelli@dm.uba.ar}
\author[E. Hern\'andez]{Eugenio Hern\'andez}\email{eugenio.hernandez@uam.es}
\author[U. Molter]{Ursula Molter}\email{umolter@dm.uba.ar}

\address[D. Barbieri and E. Hern\'andez]{Departamento de Matem\'aticas, Universidad Aut\'onoma de Madrid, 28049, Madrid, Spain}
\address[C. Cabrelli and U. Molter]{Departamento de Matem\'atica, Universidad de Buenos Aires, and Instituto de Matem\'atica ``Luis Santal\'o'' (IMAS-CONICET-UBA), 1428 Buenos Aires, Argentina}

\date{\today}



\begin{document}

\allowdisplaybreaks[2]

\begin{abstract}
In this article we study the structure   of $\Gamma$-invariant spaces of $L^2(\RR)$. Here $\RR$ is a second countable  LCA group.  The invariance is with respect to the action of $\Gamma,$ a non commutative group in the form of a semidirect product of a discrete cocompact subgroup of $\RR$ and a group of automorphisms. This class includes in particular most of the crystallographic groups.
We obtain a complete characterization of $\Gamma$-invariant subspaces in terms of range functions associated to shift-invariant spaces. We also define a new notion of range function adapted to the $\Gamma$-invariance and construct Parseval frames of orbits of some elements in the subspace, under the group action.
These results are then applied to prove the existence and construction of a $\Gamma$-invariant subspace that best approximates a set of functional data in $L^2(\RR)$.
This is very relevant in applications since in the euclidean case, $\Gamma$-invariant subspaces are invariant under rigid movements, a very sought feature in models for signal processing.

\

{\scshape R\'esum\'e}. Dans cet article nous \'etudions la structure des espaces 
$\Gamma$-invariants de $L^2(\RR)$, o\`u $\RR$ est un groupe ab\'elien 
localement compact \`a base denombrable. L'invariance est consider\'ee par rapport \`a l'action de $\Gamma$, un 
groupe non commutatif qui est le produit semi-direct d'un sous-groupe 
co-compact de $\RR$ et d'un groupe d'automorphismes.
Cette classe comprend notamment la plupart des groupes crystallographiques.
Nous obtenons une caract\'erisation compl\`ete des sous-espaces 
$\Gamma$-invariants en termes de fonctions rang associ\'ees aux espaces 
invariants par translations.
Nous d\'efinissons \'egalement une nouvelle notion de fonction rang 
adapt\'ee \`a la $\Gamma$-invariance et construisons des 
frames de Parseval \`a partir des orbites de certains \'el\'ements du sous-espace, sous l'action du groupe.
Ces r\'esultats sont ensuite appliqu\'es pour prouver l'existence et la 
construction d'un sous-espace $\Gamma$-invariant qui donne la meilleure 
approximation d'un ensemble de donn\'ees fonctionnelles dans $L^2(\RR)$.
Ceci est tr\`es pertinent pour les applications car, dans le cas 
Euclidien, les espaces  consid\'er\'es sont invariants sous l'action des mouvements rigides, une caract\'eristique tr\`es recherch\'ee dans les mod\`eles de traitement du signal.
\end{abstract}

\maketitle

\section{Introduction}

Shift-invariant spaces are good models for signals and images and have been used in many applications. They are the core spaces for multiresolution analysis and as such they are used in image compression \cite{MallatBook}. Also, they have been used in the theory of approximation \cite{dBDR94, ACHM2007}. They are subspaces invariant under translations along a lattice.

When dealing with large amounts of functional data, a common practice in applications is to assume some hypotheses on the data to use the model at hand. However it seems more natural, given a data set, to try to find a subspace that best fits the data. One can assume that our data are low dimensional in nature, that is, they belong to a ``small'' shift invariant space,  however due to possible perturbation or noise they become high dimensional.

Recently there were positive results in finding a solution to this problem (i.e. proving the existence of optimal shift invariant spaces for a given data), see \cite{ACHM2007}, \cite{CM2016}.
In  these works the class of subspaces over which the optimization is done are just translation invariant and lack additional invariances, such as rotational, that are crucial in applications to image analysis \cite{JJBOD2018} and mathematical modeling of vision \cite{CS2006}, and provide relevant geometric structures for pattern recognition and classification with neural networks \cite{SLCDV98, Mallat2012, AM2013, BLVEPD}.

In this note we consider this case in greater generality. We study the approximation problem for subspaces that are invariant under the action of a discrete locally compact group $\Gamma$, not necessarily commutative, with some hypotheses. This class in particular includes the crystallographic groups that split (see definition \ref{split}). So our spaces become invariant under rigid movements, which is a desirable property in applications. One recent application of the present results to datasets of digital images appeared in \cite{BCHMSPIE2019}. This approach turns out to be mathematically very challenging and requires many different techniques such as fiberization, gramian analysis, frame theory and group representation methods. 

To obtain our results we first need to study in depth the structure of $\Gamma$-invariant subspaces. In this direction, we obtain a characterization in terms of range functions associated to shift invariant spaces. We further introduce the notion of $\Gamma$-invariant range function, and obtain another characterization of $\Gamma$-invariant subspaces. We also prove that these $\Gamma$-invariant spaces always have a tight frame formed by orbits of some generators under the group action.

Our main result is the existence, for any given set of data, of a $\Gamma$-invariant subspace that minimizes a certain functional. We also obtain the exact value of the error in the approximation  and a formula for a set of generators of the best approximation subspace.

We describe now precisely our setting and the main results of this paper.
Let $\RR$ be a second countable locally compact abelian (LCA) group, such as $\mathbb R^d$, $\mathbb Z^d$ or $\mathbb{T}^d.$ We are interested in studying properties of the subspaces of $L^2(\RR)$ that are invariant under the action of a non commutative group that takes the form of a semidirect product $\Gamma = \ZZ \rtimes G $. Here $\ZZ$ is a discrete cocompact subgroup of $\RR$ and $G$ is a discrete and countable group that acts on $\RR$ by continuous automorphisms preserving $\ZZ$ (see Section \ref{Sec:preliminaries} for more details).

For $f\in L^2(\RR), k\in \ZZ, g\in G$ and $x\in \RR$ let $\tr{k} f(x)=f(x-k)$ and $ \rot{g}f(x)=f(g^{-1}x)$ be unitary representations of $\ZZ$ and $G$ respectively in $L^2(\RR).$ It can be proved that $(k,g) \mapsto \tr{k} \rot{g}$ is a unitary representation of $\Gamma= \ZZ \rtimes G $ in  $L^2(\RR).$

Important examples of $\Gamma$-invariant subspaces (that is, subspaces of $L^2(\RR)$ invariant under the action of the unitary representation $\tr{k}\rot{g}$) are the ones generated by a finite family $\A = \{\varphi_i\}_{i = 1}^n$ of elements of $L^2(\RR)$. Such spaces are defined as
$$
S_\Gamma(\A) := \ol{\vsp}\{\tr{k}\rot{g}\varphi : k \in \ZZ, g \in G, \varphi \in \A\}.
$$
Notice that $S_\Gamma(\A)$ is also $\ZZ$-invariant (or shift invariant), that is, invariant under the action of $T_k, k\in \ZZ.$
For these spaces, it is possible to construct a family of generators with a special desirable property.

\begin{TheoA}
For any $\Gamma$-invariant space generated by a finite family $\A$, there exists another finite family $\Psi$, with the same cardinality, which generates the same $\Gamma$-invariant space and whose $\Gamma$ orbits $\{\tr{k}\rot{g}\psi : k \in \ZZ, g \in G, \psi \in \Psi\}$ form a Parseval frame of $S_\Gamma(\Phi)$.
\end{TheoA}

To prove this result (see Theorem \ref{theo:Parseval}) we use the concept of measurable range function $\J_V$ associated to a $\ZZ$-invariant subspace $V$ of $L^2(\RR)$
and the fiberization mapping $\T$.
These techniques date back to the work \cite{Hel1964} for multiplicative shift-invariant spaces of $L^2(\mathbb{T})$ and to the works \cite{dDR1994} and \cite{Bow2000} for shift-invariant spaces of $L^2(\R^d)$. The extension to LCA groups, that we need in this paper, can be found in \cite{CP2010}.

We characterize those $\ZZ$-invariant spaces $V$ which are also $\Gamma$-invariant, in terms of the structure of the fiber spaces $\J_V(\omega)$, as follows.
\begin{TheoB}
For $g \in G$, let $\Pi(g) = \T \rot{g} \T^{-1}$. If $V$ is a $\ZZ$-invariant space, then it is also $\Gamma$-invariant if and only if
$$
\J_V(\omega) = \Pi(g) \J_V(\omega) \quad \textnormal{a.e.} \ \omega \in \wh{\RR}, \forall \ g \in G.
$$
\end{TheoB}
For the proof of this theorem, see Theorem \ref{theo:rangecovariance}.

With these tools, we are able to address the problem of approximation of a given dataset in $L^2(\RR)$ by $\Gamma$-invariant subspaces, whenever the action of the group $\Lambda^\bot \rtimes G$ on $\wh{\RR}$ admits a Borel section. Let $\data = \{f_1,\dots,f_m\} \subset L^2(\RR)$ be a set of functional data. For a closed subspace  $\V \subset L^2(\RR)$ 
define the error functional
$$
\E[\V; \data] = \sum_{i=1}^m \|f_i - \Proj_{\V}f_i\|^2_{L^2(\RR)}.
$$
We provide a constructive proof of the following result.
\begin{TheoC}
Assume that there exists a Borel section of $\wh{\RR}/(\Lambda^\bot \rtimes G)$, and let $\data = \{f_1,\dots,f_m\} \subset L^2(\RR)$. For any $1 \leq \kappa \leq m$ there exists $\Psi \subset L^2(\RR)$ of at most $\kappa$ elements  such that
$$
\E[S_{\Gamma}(\Psi); \data] = \min \{\E[S_{\Gamma}(\A);\data] : \A \subset L^2(\RR) , \ \# \A \leq \kappa \}.
$$
Moreover, the set $\{\tr{k}\rot{g}\psi \, : \, k \in \ZZ, g \in G, \psi \in \Psi \}$ is a Parseval frame of $S_\Gamma(\Psi)$, and the error $\E[S_{\Gamma}(\Psi); \data]$ can be related to the eigenvalues of the Gramian of $\DG = \{R(g)f_i \, : \, g \in G, i = 1, \dots, m\}$.
\end{TheoC}
For a full statement and a proof, see Theorem \ref{theo:main}.
This theorem includes as special cases Theorems 2.1 and 2.3 in \cite{ACHM2007}, when $\RR = \R^d$ and $G=\{ e\}$, so that $\Gamma = \ZZ \rtimes \{e\} = \ZZ.$

The paper is organized as follows: In Section \ref{Sec:preliminaries} we set the notation and give  the necessary definitions. Section \ref{Sec:structure} is devoted to study the structure of $\Gamma$-invariant spaces. In particular we prove Theorem A and two characterizations of these spaces in terms of range functions (Theorem B) and the pre-Gramian. In Section \ref{Sec:Section} we give conditions for the existence of a Borel section for the action of $\Gamma$. The existence of a Borel section allows us to introduce a notion of measurable $\Gamma$-invariant range function, and to obtain another characterization of $\Gamma$-invariant subspaces. The proof of Theorem C is given in Section \ref{sec:approx} (see Theorem \ref{theo:main}). In this proof we use techniques developed in previous sections and the SVD decomposition used in \cite{ACHM2007}. A particularly important example covered in our situation are the crystallographic groups (see Section \ref{Sec:example}).

\

\noindent
{\bf Acknowledgements}: 
This project has received funding from the European Un\-ion's Horizon 2020 research and innovation programme under the Marie Sk\l odowska-Curie grant agreement No 777822. In addition, D. Barbieri and E. Hern\'andez were supported by Grant MTM2016-76566-P (Ministerio de Ciencia, Innovaci\'on y Universidades, Spain). C. Cabrelli and U. Molter were supported by Grants UBACyT 20020170100430BA (University of Buenos Aires), PIP11220150100355 (CONICET) and PICT 2014-1480 (Secretary of Science and Technology from Argentina).

\section{Preliminaries} \label{Sec:preliminaries}

Let $\RR$ be a second countable LCA group and let $\ZZ$ be a uniform lattice of $\RR$ (i.e. $\ZZ$ discrete such that $\RR/\ZZ$ is compact). Let $G$ be a discrete and at most countable group that acts on $\RR$ by continuous automorphisms, and let us denote this action by $(g,x) \mapsto g x \in \RR$ for $g \in G$ and $x \in \RR$. 
We will also denote by $\card$ the cardinality of $G$, which can be either finite or infinite.

Let $\wh{\RR}$ be the Pontryagin dual of $\RR$. We denote the Pontryagin duality indistinctly by
$
\langle \xi , x \rangle = e^{2\pi i \xi . x} \ , \quad \xi \in \wh{\RR} , x \in \RR
$
and we denote by $\ZZdual = \{ \ell \in \wh{\RR} \, : \, \langle \ell, k \rangle = 1 \ \forall \, k \in \ZZ\}$ the annihilator of $\ZZ$.
The action of $G$ on $\RR$ induces an action of $G$ on $\wh{\RR}$ by duality:
$
\langle g^* \xi , x \rangle := \langle \xi, g x\rangle.
$
Observe that this action satisfies
\begin{equation} \label{Eq:star}
g_1^* g_2^* = (g_2 g_1)^* \quad \mbox{for all} \ g_1, g_2 \in G.
\end{equation}
Assume that that the action of $G$ on $\RR$ preserves $\ZZ$, that is $g \ZZ = \ZZ$ for all $g \in G$, and observe that this is equivalent to require that the dual action of $G$ on $\wh{\RR}$ preserves $\ZZ^\bot$.

We will need to let $G$ act on a Borel section $\Omega \subset \wh{\RR}$ for the action of $\ZZ^\bot$ on $\wh{\RR}$.
\begin{defi}
The action of $G$ on the quotient group $\wh{\RR}/\ZZ^\bot$ is
$$
g^*[\xi] := [g^* \xi] \ , \quad \xi \in \wh{\RR}, \quad g \in G,
$$
where $[\xi]$ is the class of $\xi$ in $\wh{\RR}/\ZZ^\bot$. If $\Omega \subset \wh{\RR}$ is a Borel section of $\wh{\RR}/\ZZ^\bot$, we will denote the action of $G$ on $\Omega$ accordingly, keeping for simplicity the same notation $(g,\omega) \mapsto g^*\omega$ for $g \in G$ and $\omega \in \Omega$.
\end{defi}
Notice that this coincides with the dual action of $G$ on $\wh{\RR}$ only if $\Omega$ is an invariant subset of $\wh{\RR}$ for that action.

Since the action of $G$ preserves $\ZZ$, we can define the semidirect product $\Gamma = \ZZ \rtimes G = \{(k,g) \, : \, k \in \ZZ, g \in G\}$, with composition law
$$
(k,g) \cdot (k',g') = (k + gk', gg').
$$
The action of $\Gamma$ on $\RR$ given by
$$
\gamma x = gx + k \ , \quad \gamma = (k,g) \in \Gamma \, , \ x \in \RR
$$
will provide the symmetry with respect to which we will study invariance.

We will denote the Haar measure of $\RR$ of a measurable set $E \subset \RR$ by $|E|$. 
Since $g\ZZ = \ZZ$, then the action of $G$ on $\RR$ preserves the Haar measure, i.e.
$$
|gE| = |E| \ , \quad \forall \ E \subset \RR \textnormal{ measurable}, \quad \forall \ g \in G .
$$
Note also that the invariance of the Haar measure of $\RR$ under the action of $G$ implies that the same property holds in $\wh{\RR}$ (see e.g. \cite[Lemma 3]{BHM19}).

\subsection{Fundamental operators} \label{Sec:fundamental}

For $\Hil$ a Hilbert space, we will denote by $\B(\Hil)$ the bounded operators on $\Hil$.
We use the following notation for the Fourier transform of $f \in L^1(\RR)$:
$$
\F f (\xi) = \wh{f}(\xi) = \int_{\RR} e^{-2\pi i \xi . x} f(x) dx = \int_{\RR} \ol{\langle \xi , x \rangle} f(x) dx ,
$$
and it can be extended to an isometry of $L^2(\RR)$ by density.

Let $T : \ZZ \to \U(L^2(\RR))$ and $R : G \to \U(L^2(\RR))$ be the representations
$$
\tr{k}f(x) = f(x - k) \, , \ \rot{g}f(x) = f(g^{-1}x) \ , \quad f \in L^2(\RR).
$$
Since $\rot{g}\tr{k} = \tr{gk}\rot{g}$, the map $(k,g) \mapsto \tr{k}\rot{g}$ defines a unitary representation of $\Gamma = \ZZ \rtimes G$ on $L^2(\RR)$.

It is easy to prove that for all $f \in L^2(\RR)$ and all $(k,g) \in \Gamma$ we have
\begin{equation}\label{eqlem:Fourier}
\wh{\tr{k} f}(\xi) = e^{-2\pi i \xi . k} \wh{f}(\xi) \ , \quad
\wh{\rot{g} f}(\xi) = \wh{f}(g^* \xi).
\end{equation}

\begin{defi}\label{def:mappingT}
For $f\in L^2(\RR)$ and $\omega \in \wh{\RR}$ define the mapping $\T$ as
\begin{equation}\label{eqn:mappingT}
\T[f] (\omega) = \{\wh{f}(\omega + s)\}_{s \in \ZZdual}\,, \quad f\in L^2(\RR), \ \omega \in \wh{\RR}.
\end{equation}
\end{defi}

The following result can be found in \cite[Proposition 3.3]{CP2010}.

\begin{lem}\label{lem:gram}
Let $\Omega \subset \wh{\RR}$ be a Borel section of $\,\wh{\RR} / \ZZdual \approx \wh{\ZZ}$
The periodization mapping  $\T$ defined above is an isometric isomorphism between the Hilbert spaces $L^2(\RR)$ and $ L^2(\Omega,\ell_2(\ZZdual))$.
\end{lem}

\begin{defi}\label{def:gram}
For a finite family $\A = \{\varphi_i\}_{i = 1}^n \subset L^2(\RR)$, we denote by $\TT_\A $ the pre-Gramian\footnote{Field over $\Omega$ of synthesis operators in $\ell_2(\ZZdual)$ with respect to the families $\{\T[\varphi_i](\omega)\}_{i = 1}^n$.} of $\A$,  $\TT_\A \in L^2(\Omega,\B(\C^n,\ell_2(\ZZdual)))$. It is given by
\begin{equation}\label{Eq:pre-gramian}
\TT_\A (\omega) c = \sum_{j = 1}^n \T[\varphi_j](\omega) c_j \ , \quad c = (c_j)_{j=1}^n \in \C^n\,, \ \omega \in \Omega.
\end{equation}
\end{defi}
Note that the pre-Gramian can be seen as the (possibly infinite) matrix
$$
\TT_\A (\omega) = \left(
\begin{array}{ccc}
\vdots &  & \vdots \\
\T[\varphi_1](\omega) & \dots & \T[\varphi_n](\omega)\\
\vdots &  & \vdots 
\end{array}
\right) \,,
$$
and its adjoint $\TT_\A^* \in L^2(\Omega,\B(\ell_2(\ZZdual),\C^n))$ is given by
$$
\TT_\A^*(\omega) v = \{\langle v, \T[\varphi_i](\omega)\rangle_{\ell_2(\ZZdual)}\}_{i = 1}^n \ , \quad v \in \ell_2(\ZZdual).
$$

\begin{defi}\label{def:gramian}
The Gramian of $\A$ is the operator $\G_\A = \TT_\A^*\TT_\A \in L^1(\Omega,\C^{n \times n})$, that reads
\begin{align}
\G_\A(\omega) c&  = \TT_\A^*(\omega)\TT_\A(\omega) c = \TT_\A^*(\omega)\left(\sum_{j = 1}^n \T[\varphi_j](\omega)c_j \right )\nonumber\\
& = \left\{\sum_{j = 1}^n \langle \T[\varphi_j](\omega),\T[\varphi_i](\omega)\rangle_{\ell_2(\ZZdual)} c_j \right\}_{i=1}^n, \  \quad c = (c_j)_{j=1}^n \in \C^n\,.\label{Eq:gramian}
\end{align}
\end{defi}
Note that, if $\G_\A(\omega)_{i,j}$ denotes the $(i,j)$ element of the $n\times n$ matrix of $\G_\A(\omega)$, $i,j \in \{1,\dots,n\}$,  we have
$$
\G_\A(\omega)_{i,j} = \langle \T[\varphi_j](\omega),\T[\varphi_i](\omega)\rangle_{\ell_2(\ZZdual)} = \sum_{s \in \ZZdual} \wh{\varphi_j}(\omega + s) \ol{\wh{\varphi_i}(\omega + s)}.
$$

We will need to intertwine the representation $(k,g) \mapsto \tr{k}\rot{g}$ of $\Gamma$ on $L^2(\RR)$ with the isomorphism $\T$. This will be done in Lemma \ref{lem:intertwining}. Before, we introduce the following notation.

\begin{defi}\label{def:rpeque}
We denote by $r : G \to \U(\ell_2(\ZZdual))$ the representation given by
$$
(r_g(a))(s) = a(g^* s) \ , \quad a \in \ell_2(\ZZdual) , s \in \ZZdual.
$$
\end{defi}

\begin{lem}\label{lem:intertwining}
For $f\in  L^2(\RR)$, $(k,g)\in \Gamma$, and $\omega \in \Omega$,
$$
\T[\tr{k} \rot{g} f] (\omega) = e^{-2\pi i \omega.k} r_g \T[f](g^*\omega).
$$
\end{lem}
\begin{proof}
Using (\ref{eqlem:Fourier}),
\begin{align*}
\T[\tr{k} \rot{g} f] (\omega) & = \{\F(\tr{k}\rot{g}f)(\omega + s)\}_{s \in \ZZdual} = \{e^{-2\pi i \omega.k}\F(\rot{g}f)(\omega + s)\}_{s \in \ZZdual}\\
& = e^{-2\pi i \omega.k}\{\F(f)(g^*\omega + g^* s))\}_{s \in \ZZdual}\\
& = e^{-2\pi i \omega.k} r_g\{\F(f)(g^*\omega + s)\}_{s \in \ZZdual} = e^{-2\pi i \omega.k} r_g \T[f](g^*\omega). \qedhere
\end{align*}
\end{proof}
\begin{defi}\label{def:Pi}
For $g\in G$, we denote by $\Pi(g) = \T \rot{g} \T^{-1}$. Then, $\Pi$ is a unitary representation of $G$ on the Hilbert space $L^2(\Omega,\ell_2(\ZZdual))$ and, by Lemma \ref{lem:intertwining}, it reads
\begin{equation}\label{eq:intertwinedrep}
\Pi(g) F(\omega) = r_g(F(g^*\omega)) \, , \quad F \in L^2(\Omega,\ell_2(\ZZdual)).
\end{equation}
\end{defi}

\subsection{Shift-invariant spaces}

A closed subspace $\V \subset L^2(\RR)$ is $\ZZ$-invariant (or shift-invariant by $\ZZ$) if $\tr{k}\V \subset \V$ for all $k \in \ZZ$. 
For $\A \subset L^2(\RR)$ a countable family, we will write 
$$
S(\A):= \ol{\vsp}\{\tr{k}\varphi \, : \, k \in \ZZ, \varphi \in \A\}\,.
$$
Since $L^2(\RR)$ is separable, if $\V$ is a $\ZZ$-invariant subspace of $L^2(\RR)$, there exist a countable set $\A \subset L^2(\RR) $ such that $\V = S(\A).$

The study of the structure of shift-invariant spaces can be done in terms of the so called range function, defined as follows.

\begin{defi}\label{rf-def}
Let $\Omega \subset \wh{\RR}$ be a Borel section of $\wh{\RR} / \ZZdual \approx \wh{\ZZ}$.
A range function is a map
$$
\J : \Omega \to \{\textnormal{closed subspaces of } \ell_2(\ZZdual)\} .
$$
A range function $\J$ is said to be measurable if the family $\pj_{\J(\omega)} \in \B(\ell_2(\ZZdual))$ of orthogonal projections onto $\J(\omega)$ is measurable.
Given a range function $\J$, we will denote by $M_\J$ the closed subspace of $L^2(\Omega,\ell_2(\ZZdual))$ given by
$$
M_\J = \{F \in L^2(\Omega,\ell_2(\ZZdual)) \, : \, F(\omega) \in \J(\omega) \ \textnormal{a.e. } \omega \in \Omega\} .
$$
\end{defi}

We state below the results we need in the sequel. Their proofs can be found in \cite{CP2010}.

\begin{lem} \label{Lem:range}
Let $\J$ be a measurable range function, let $\Proj_{M_\J} \in \B(L^2(\Omega,\ell_2(\ZZdual)))$ be the orthogonal projection onto $M_\J$ and let $\{\pj_{\J(\omega)}\}_{\omega \in \Omega} \subset \B(\ell_2(\ZZdual))$ be the measurable field of orthogonal projections onto $\{\J(\omega)\}_{\omega \in \Omega}$. Then
$$
(\Proj_{M_\J}F)(\omega) = \pj_{\J(\omega)}(F(\omega)) \quad \forall \ F \in L^2(\Omega,\ell_2(\ZZdual)) \ , \quad \textnormal{a.e. } \omega \in \Omega .
$$
\end{lem}

\begin{theo}\label{Th:Helson}
Let $\V$ be a closed subspace of $L^2(\RR)$ and $\T$ the map of Definition \ref{def:mappingT}. The subspace $\V$ is $\Lambda$-invariant if and only if there exists a unique measurable range function $\J_\V$ such that $\T(\V)=M_{\J_\V}$.

Moreover, if $\V=S(\A)$ for some countable set $\A$ of $L^2(\RR)$, the measurable range function associated to $S(\A)$ satisfies
$$
\J_\V(\omega) =\ol{\vsp}\{\T[\varphi](\omega) \, : \, \varphi \in \A\}, \quad a.e. \ \omega \in \Omega.
$$
\end{theo}

\begin{rem}\label{range-unicidad}
Uniqueness of the the measurable range function in Theorem \ref{Th:Helson} is understood in the following sense: two range functions $\J_1$ and $\J_2$ are equal if $\J_1(\omega) = \J_2(\omega)\  a. e. \ \omega \in \Omega.$ The uniqueness of the measurable range function associated to a closed $\Lambda$-invariant subspace $\V$ of $L^2(\RR)$ in Theorem \ref{Th:Helson} is a consequence of Lemma \ref{Lem:range} (for the proof see \cite[Lemma 3.11]{CP2010}).
\end{rem}

\subsection{Range function and SNAG theorem}

It is worth to observe that the given formulation of invariance by translations of LCA groups, whose approach in terms of range functions dates back to \cite{Hel1964}, can be restated in terms of a generalization of a well-known result by Stone (see e.g. \cite[Theorem 4.45]{Folland2015}), also known as SNAG Theorem. This result states that for any unitary representation $\pi$ of an LCA group $\ZZ$ on a Hilbert space $\mathcal{H}$ there exists a projection valued measure $E_\pi$ of $\mathcal{H}$, defined on $\widehat{\ZZ}$, such that each $\pi(k)$, for $k \in \ZZ$, can be written as
$$
\pi(k) = \int_{\widehat{\ZZ}} e^{2\pi i \omega.k} dE_\pi(\omega).
$$

The relationship between this decomposition and the construction of generalized multiresolution analysis was pointed out in \cite{Packer2004}, while the role of the SNAG Theorem for the study of reproducing systems in invariant spaces under general unitary representations of LCA groups was clarified in \cite{HSWW10}. In that paper, the authors define the unitary representation $\pi$ to be dual integrable if, for all $\varphi, \psi \in \mathcal{H}$, the measure given by $\langle E_\pi(B)\varphi,\psi\rangle_{\mathcal{H}}$ on Borel subsets $B \subset \wh{\ZZ}$ is absolutely continuous with respect to the Haar measure on $\widehat{\ZZ}$. They call the corresponding Radon-Nikodym derivative \emph{bracket map}, denoted by $[\varphi,\psi]$, and prove that the class of dual integrable unitary representations coincides with that of square integrable ones. These results were later extended to non abelian discrete groups in \cite{BHP2015}. When specialized to the case under study, described in the previous sections, the SNAG Theorem states that there exists a projection valued measure $E_T$ of $L^2(\RR)$, defined on $\Omega$, such that each $\tr{k}$ can be written as $\tr{k} = \int_{\Omega} e^{2\pi i \omega.k} dE_T(\omega)$. This representation is dual integrable, and an explicit expression for the bracket map for $\varphi, \psi \in L^2(\RR)$ is $[\varphi,\psi](\omega) = \sum_{s \in \ZZdual} \widehat{\varphi}(\omega + s)\ol{\widehat{\psi}(\omega + s)} = \langle\T[\varphi](\omega),\T[\psi](\omega)\rangle_{\ell_2(\ZZdual)}$. Indeed, we have that
$$
\langle T_k\varphi,\psi\rangle_{\mathcal{H}} = \int_\Omega e^{2\pi i \omega.k} [\varphi,\psi](\omega) d\omega.
$$

On the other hand, range functions can be formulated in terms of direct integrals. In \cite{BR15} the authors  observe that, for a $\ZZ$-invariant subspace $\V \subset L^2(\RR)$, and denoting by $\T$ the isomorphism (\ref{eqn:mappingT}), it holds that
$\displaystyle
\T(\V) = \int_{\Omega}^\oplus \J_\V(\omega) d\omega.
$
Thus, by Lemma \ref{Lem:range}, we can write the orthogonal projection of $L^2(\Omega,\ell_2(\ZZdual))$ onto $\T(\V)$ as
\begin{equation}\label{eq:RangeProjection}
\Proj_{\T(\V)} = \int_{\Omega}^\oplus \pj_{\J_\V(\omega)} d\omega.
\end{equation}

The relationship between this direct integral decomposition and the SNAG decomposition of $\tr{k}$ can be made explicit as follows. Let $m = \T \tr{} \T^{-1}$ be the unitary representation of $\ZZ$ on $L^2(\Omega,\ell_2(\ZZdual))$ obtained by intertwining $\tr{}$ with $\T$. It reads
$$
m(k) F(\omega) = e^{2\pi i \omega.k} F(\omega) \ , \quad F \in L^2(\Omega,\ell_2(\ZZdual)).
$$
The projection valued measure $E_m$ for $m$ given by the SNAG Theorem is then related to the one for $T$ by
$$
E_m(B) = \T E_T(B) \T^{-1} \ , \quad B \subset \Omega \ \textnormal{ Borel}.
$$
Thus we have that for all Borel sets $B \subset \Omega$
$$
\langle E_m(B)F,H\rangle_{L^2(\Omega,\ell_2(\ZZdual))} = \int_B \langle F(\omega),H(\omega)\rangle_{\ell_2(\ZZdual)} d\omega \ , \quad F, H \in L^2(\Omega,\ell_2(\ZZdual))
$$
so that $E_m$ reads explicitly $E_m(B) F(\omega) = \one{B}(\omega)F(\omega)$. Let now $\V \subset L^2(\RR)$ be a $\ZZ$-invariant subspace, and let $m^\V = m \Proj_{\T(\V)}$ be the subrepresentation of $m$ on the invariant subspace $\T(\V)$. Its projection valued measure $E_{m^\V}$ of $\T(\V)$, defined on $\Omega$, obtained by the SNAG Theorem, extended to $L^2(\Omega,\ell_2(\ZZdual))$, reads then
$$
E_{m^\V}(B) = E_m (B) \Proj_{\T(\V)} = \Proj_{\T(\V)} E_m (B) \ , \quad B \subset \Omega \ \textnormal{ Borel}
$$
where the second identity is due to invariance. Thus, combining the action of $E_m$ with (\ref{eq:RangeProjection}), we have that $E_{m^\V}$ can be explicitly written in terms of the range function of $\V$ as
$$
E_{m^\V}(B) = \int_{B}^\oplus \pj_{\J_\V(\omega)} d\omega  \ , \quad B \subset \Omega \ \textnormal{ Borel}.
$$

\subsection{Examples} \label{Sec:example}
\

The motivational examples for our setting are the crystallographic groups that split.
The beginning of the study of crystallographic, or crystal, groups is related to the problem of tiling 
$\R^d$ by rigid motions. That is, to find a closed bounded set $P$ that covers the space by rigid movements. More precisely we have the following definition.
\begin{defi}\label{defgrupcris}
A crystal group $\Gamma$ is a discrete subgroup of the group of isometries of $\R^d$ $(\Gamma\subset \mathrm{Isom}(\R^d)\,)$ that has a closed and bounded Borel section $P$, that is,
\begin{itemize}
\item[1.] $\displaystyle\bigcup_{\gamma\in\Gamma} \gamma P = \R^d.$
\item[2.] If $\gamma=\gamma'$, then $|\gamma P \cap\gamma' P | = 0 $.
\end{itemize}
\end{defi}

A fundamental theorem for crystal groups, due to Bieberbach \cite{Bie1910}, is
\begin{theo}[Bieberbach] Let $\Gamma$ be a crystal subgroup of ${\rm Isom}(\mathbb{R}^{d})$ and let us denote by ${\rm Trans}(\R^d)$ the translations of $\R^d$. Then
\begin{itemize}
\item[1.]
$\Lambda :=\Gamma \cap {\rm Trans}(\R^d)$ is a finitely generated abelian group of rank $d$ which spans ${\rm Trans}(\R^d)$
\item[2.] the linear part of the symmetries of $\Gamma$, the point group of $\Gamma$, is finite and is isomorphic to $\Gamma/\Lambda$.
\end{itemize}
\end{theo}
A subclass of the crystal groups are those that split.
\begin{defi}\label{split}
We say that a crystal group $\Gamma$  splits, if it is the  semidirect product $\Gamma = \ZZ \rtimes G$ of a finite group $G$ and a uniform lattice $\ZZ$ of $\R^d$.
\end{defi}
It can be shown that any crystal group can be embedded in a crystal group that splits, so most of the results for splitting crystal groups can be directly transferred to the general case. For general results on crystal groups see \cite{Bie1910, Far1981, GS1987}.

\vspace{2ex}

Another example, very different in nature than the crystal groups is the group of translations and shears: Consider $\RR = \R^2$, $\ZZ = \Z^2$ and $G = \{\binom{1 \ s}{0 \ 1} \, : \, s \in \Z\}$, which preserves the lattice $\Z^2$. In this case, since $G$ is infinite, the action of $\Gamma$ on $\RR$ does not admit a Borel section (see Proposition \ref{prop:Effros}), so this group is not a crystallographic group, however all results of Section \ref{Sec:structure} still hold in this case.

\section{The structure of $\Gamma$-invariant spaces}\label{Sec:structure}

In this section we study the structure of closed subspaces of $L^2(\RR)$ that are invariant under the action of $\Gamma$. Recall that $\Gamma =  \ZZ \rtimes G$ is the semidirect product of a uniform lattice $\Lambda$ in $\RR$ and a discrete and countable group $G$ that acts  on $\RR$ by continuous invertible automorphisms.
In addition, we also assume that $g\Lambda = \Lambda$ for all
$g\in G$, which implies that the Haar measure of $\RR$ is invariant under the action of $G$ (see Section \ref{Sec:preliminaries}).

A closed subspace $V \subset L^2(\RR)$ is $\Gamma$-invariant if $\tr{k} \rot{g} V \subset V$ for all $(k,g) \in \Gamma.$
Equivalently, $V$ is $\Gamma$-invariant if
$$
f \in V \ \Rightarrow \ \tr{k}f \in V \ \ \forall \ k \in \ZZ \, , \ \textnormal{and} \ \rot{g}f \in V \ \ \forall \ g \in G.
$$

\begin{rem} \label{Remark2-6}
Observe that $V$ is $\Gamma$-invariant if and only if
$$
V \ \textrm{is shift-invariant, and} \ \Pi(g)\T[V] \subset \T[V] \ \forall \ g \in G
$$
where $\Pi$ is given in Definition \ref{def:Pi}
\end{rem}

For an at most countable family $\A \subset L^2(\RR)$ , we will write
$$
S_\Gamma(\A) := \ol{\vsp}\{\tr{k}\rot{g}\varphi \, : \, k \in \ZZ, g \in G, \varphi \in \A\}.
$$
$S_\Gamma(\A)$ is a $\Gamma$-invariant space and the set $\A$ is called a set of generators. Note that, since $\tr{k}\rot{g} = \rot{g}\tr{g^{-1}k}$, we also have that
$$
S_\Gamma(\A) = \ol{\vsp}\{\rot{g}\tr{k}\varphi : \, k \in \ZZ, g \in G, \varphi \in \A\}.
$$
Since $L^2(\RR)$ is separable, if $V$ is a $\Gamma$-invariant subspace of $L^2(\RR)$, there always exists a countable set $\A \subset L^2(\RR) $ such that $V = S_\Gamma(\A).$

\begin{defi}
Let $V$ be a $\Gamma$-invariant subspace of $L^2(\RR)$. We denote by $\len(V)$, the length of $V$, the minimum number of generators of $V$:
$$
\len(V) = \min \{n: \exists\,\A = \{\varphi_1, \dots,\varphi_n\}: V = S_\Gamma(\A) \}.
$$
If $V$ does not have a finite number of generators we set $\len(V) =\infty.$
\end{defi}

Our first result is the characterization of $\Gamma$-invariant closed subspaces in terms of a covariance property of the range function associated to its $\Lambda$-invariant subspace.

\begin{theo}\label{theo:rangecovariance}
A closed subspace $V$ of $L^2(\RR)$ is $\Gamma$-invariant if and only if it is $\Lambda$-invariant (shift-invariant by $\ZZ$) and its range function $\J_V =\J$ satisfies
\begin{equation}\label{eq:rangecovariance}
\J(g^*\omega) = r_{g^{-1}}\ \J(\omega) \, , \ \textnormal{a.e.} \ \omega \in \Omega \, , \ \forall g \in G.
\end{equation}
\end{theo}
\begin{proof}
Assume first that $V$ is $\Lambda$-invariant and satisfies 
\eqref{eq:rangecovariance}. By Theorem \ref{Th:Helson},
\begin{equation} \label{Eq:2}
  F\in \T(V) \Longleftrightarrow F(\omega) \in \J(\omega), \ a. e. \ \omega \in \Omega.
\end{equation}
By Remark \ref{Remark2-6} it is enough to show the inclusion $\Pi(g)\T(V) \subset \T(V)$ for all $g\in G.$ To show this, let $g\in G$ and $F\in \T(V).$ By \eqref{eq:intertwinedrep} and \eqref{eq:rangecovariance},
$$
\Pi(g)F(\omega) = r_g(F(g^*\omega)) \in r_g \J(g^*\omega) = \J(\omega), \ a. e. \ \omega \in \Omega.
$$
Hence, by \eqref{Eq:2}, $\Pi(g)F \in \T(V)$ as wanted.

Assume now that $V$ is $\Gamma$-invariant. By Remark \ref{Remark2-6}, $V$ is $\Lambda$-invariant and $\Pi(g)\T(V)=\T(V)$ for all $g\in G.$ By Theorem \ref{Th:Helson}, $\T(V)= M_\J.$ Observe that, for $g\in G$, since $r_g$ is a unitary operator of $\ell_2(\Lambda^\perp)$, $\J_g(\omega) := r_g(\J(g^*\omega))$ is a measurable range function. Now
$$
F\in \Pi(g)\T(V) \Longleftrightarrow \Pi(g^{-1})F \in \T(V) \Longleftrightarrow \Pi(g^{-1})F(\omega) \in \J(\omega), \ a. e. \ \omega \in \Omega.
$$
By \eqref{eq:intertwinedrep}, $F\in \Pi(g)\T(V) \Longleftrightarrow r_{g^{-1}} F((g^{-1})^*\omega)\in \J(\omega).$ Since, by \eqref{Eq:star}, $(g^{-1})^* g^* = (g g^{-1})^* = e$, we deduce
$$
F\in \Pi(g)\T(V) \Longleftrightarrow r_{g^{-1}} F(\omega) \in \J(g^* \omega) \Longleftrightarrow F(\omega) \in \J_g(\omega).
$$
Thus, $\Pi(g)\T(V)= M_{\J_g}.$ Since $\Pi(g)\T(V)=\T(V)$ we obtain $M_{\J_g} = M_\J.$ By \cite[Lemma 3.11]{CP2010}, $\J_g(\omega)= \J(\omega)\ a. e. \ \omega \in \Omega$, which is (\ref{eq:rangecovariance}).
\end{proof}

We now study the structure of the pre-Gramian and Gramian operators (see equations \eqref{Eq:pre-gramian} and \eqref{Eq:gramian}) of a set $\AG :=\{\rot{g}\varphi_i \, : \, i \in \{1,\dots,n\}, g \in G\}$, containing the $G$-orbits of the finite family $\A = \{\varphi_i\}_{i = 1}^n \subset L^2(\RR)$.

We will fix an ordering in the collection $\AG$. To do this we order the finite group $G = \{ g_1, g_2, \dots , g_\card\}$ with $g_1=e$, the identity element of the group, and then order the set $I_n\times G := \{1, 2, \dots, n\} \times G$ with the lexicographical ordering. Thus, if $(i_1, g_{j_1}), (i_2, g_{j_2}) \in I_n\times G$, we have $(i_1, g_{j_1}) < (i_2, g_{j_2})$ if and only if $i_1 < i_2$ or if $i_1 = i_2$, then $j_1 < j_2.$  In the following we will use the lexicographical ordering in $I_n\times G := \{1, 2, \dots, n\} \times G$ to simplify the notation; nevertheless, all the results hold for any other ordering.

\begin{rem}
In the rest of this paper, we will always consider a finite family $\A = \{\varphi_i\}_{i = 1}^n \subset L^2(\RR)$ and a finite group $G$. However, in case either the family $\A$, or the group $G$, or both, are countably infinite, the results of this section can be proved provided the following Bessel condition holds:
$$
\sum_{i \in \N} \sum_{k \in \Lambda} \sum_{g \in G} |\langle f, \tr{k}\rot{g}\varphi_i\rangle_{L^2(\RR)}|^2 \leq B \|f\|^2_{L^2(\RR)} \quad \forall \ f \in L^2(\RR)
$$
for some $B > 0$. Indeed, this condition is equivalent to the requirement that the pre-Gramians that we will deal with, be bounded operators. The finiteness assumptions are left in this section for both, the sake of clarity and the fact that they will hold naturally in the next sections.
\end{rem}

\begin{defi} \label{Def:lambda}
For $c : I_n \times G \to \C$, let us denote by
$$
(\lambda_g \,c)_{j,g'} = c_{j,g^{-1}g'} \ , \quad g,g' \in G, j \in I_n.
$$
\end{defi}
Note that, for each $g \in G$, we are denoting by $\lambda_g \in M_{n \card}(\C)$ the block matrix that for each $j \in I_n$ shifts the components of the vector $c(j,\cdot) \in \C^{\card}$ according to the group law of $G$. That is, $\lambda$ acts as the left regular representation of $G$ on the second variable (index) of the two-indices vector $c$, or equivalently $\lambda$ is the direct sum of $n$ copies of the left regular representation of $G$.

\begin{theo} \label{theo:preGramiancovariance}
Let $\A = \{\varphi_i\}_{i = 1}^n \subset L^2(\RR)$ be a finite family and let $\TT_{\AG}$ be the pre-Gramian of $\AG :=\{\rot{g}\varphi_i \, : \, g \in G, i \in I_n\}$. Then
\begin{equation}\label{eq:preGramiancovariance}
\TT_{\AG}(g^*\omega) = r_{g^{-1}}\TT_{\AG}(\omega)\, \lambda_g \ , \quad \textnormal{a.e. } \omega \in \Omega
\end{equation}
where $r$ is given in Definition \ref{def:rpeque} and $\lambda$ is given in Definition \ref{Def:lambda}.

Conversely, consider the set of indexes  $I_n\times G := \{1, 2, \dots, n\} \times G$ with the lexicographical order and let $\Psi = \{\psi_{i,g} \, : \, i\in I_n, g\in G \}$ be a collection of $n\card$ elements of $L^2(\RR)$. Suppose that for all $g\in G$,
\begin{equation}\label{eq:preGramiancovarianceII}
\TT_{\Psi}(g^*\omega) = r_{g^{-1}}\TT_{\Psi}(\omega)\, \lambda_g \ , \quad \textnormal{a.e. } \omega \in \Omega\,,
\end{equation}
and set $\varphi_j  = \psi_{j,e}, \,j\in I_n \,$ (with $e$ the identity in G), then $\rot{g}\varphi_j = \psi_{j,g}$ for all $(j,g) \in I_n \times G$.
\end{theo}

\begin{proof}
Note first that the composition of operators in (\ref{eq:preGramiancovariance}) is well defined. Indeed:
\begin{equation*}
\lambda_g : \C^{n \card} \to \C^{ n \card}, \quad
\TT_{\AG}(\omega)  : \C^{n \card} \to \ell_2(\ZZdual), \quad
r_{g^{-1}}\  : \ell_2(\ZZdual) \to \ell_2(\ZZdual).
\end{equation*}
Let then $c : I_n\times G \to \C$. By Definition \ref{def:gram}, and using Lemma \ref{lem:intertwining}, we get
\begin{align*}
\TT_{\AG}(g^*\omega)c & = \sum_{j = 1}^n \sum_{g' \in G} \T[\rot{g'}\varphi_j](g^*\omega)\,c_{j,g'}\\
& = \sum_{j = 1}^n \sum_{g' \in G} r_{g^{-1}}\T[\rot{g g'}\varphi_j] (\omega)\, c_{j,g'} = \sum_{j = 1}^n \sum_{g'' \in G} r_{g^{-1}}\T[\rot{g''}\varphi_j] (\omega)\, c_{j,g^{-1}g''}\\
& = \sum_{j = 1}^n \sum_{g'' \in G} r_{g^{-1}}\T[\rot{g''}\varphi_j] (\omega)\, (\lambda_g \,c)_{j,g''} = r_{g^{-1}} \TT_{\AG}(\omega) \lambda_g c \, .
\end{align*}
To prove the converse statement observe first that, by similar arguments to the ones above, we have
\begin{align*}
r_{g^{-1}}\TT_{\Psi}(\omega) \lambda_g c & = \sum_{j = 1}^n \sum_{g'' \in G} r_{g^{-1}}\T[\psi_{j,g''}] (\omega) c_{j,g^{-1}g''}\\
& = \sum_{j = 1}^n \sum_{g' \in G} r_{g^{-1}}\T[\psi_{j,gg'}] (\omega) c_{j,g'}
\end{align*}
and
\begin{align*}
\TT_{\Psi}(g^*\omega)c & = \sum_{j = 1}^n \sum_{g' \in G} \T[\psi_{j,g'}](g^*\omega)\,c_{j,g'} = \sum_{j = 1}^n \sum_{g' \in G} r_{g^{-1}}\T[\rot{g}\psi_{j,g'}](\omega)\,c_{j,g'}\,.
\end{align*}
Therefore, if (\ref{eq:preGramiancovarianceII}) holds, then
$$
\T[\psi_{j,gg'}] (\omega) = \T[\rot{g}\psi_{j,g'}](\omega) \ , \quad \forall j,g,g' \, , \ \textnormal{a.e. } \omega \in \Omega.
$$
Thus, since $\T$ is an isomorphism (see Lemma \ref{lem:gram}), $\psi_{j,gg'} = \rot{g}\psi_{j,g'}\,.$ Choose $\varphi_j = \psi_{j,\id}$, where $\id$ is the identity element of $G$. Then, $\rot{g}\varphi_j = \rot{g}\psi_{j,e} = \psi_{j,g},$ as wanted.  
\end{proof}

\begin{cor}\label{cor:Gramcovariance}
Let $\A = \{\varphi_i\}_{i = 1}^n \subset L^2(\RR)$ be a finite family and let $\G_{\AG}$ be the Gramian of $\AG = \{\rot{g}\varphi_i  : (i,g) \in I_n \times G\}$, see Definition \ref{def:gramian}. Then
\begin{equation}\label{eq:Gramiancovariance}
\G_{\AG}(g^*\omega) = \lambda_{g^{-1}}\G_{\AG}(\omega)\, \lambda_g .
\end{equation}
\end{cor}
\begin{proof}
By definition of Gramian $\G_{\AG}= \TT_{\AG}^*\TT_{\AG}$, and using (\ref{eq:preGramiancovariance}), we have
\begin{align*}
\G_{\AG}(g^*\omega) & = \TT_{\AG}^*(g^*\omega)\TT_{\AG}(g^*\omega) = \lambda_g^*\,\TT_{\AG}(\omega)^*r_{g^{-1}}^*r_{g^{-1}}(\TT_{\AG}(\omega))\,\lambda_g.
\end{align*}
The proof is concluded by noting that, since $r$ and $\lambda$ are unitary homomorphisms, $r_{g^{-1}} = r_g^{-1} = r_g^* \ \forall g \in G$, and the same is true for $\lambda$.
\end{proof}

Using these results we give in the next theorem an explicit construction of a Parseval frame of orbits for a finitely generated $\Gamma$-invariant space. Existence of Parseval frames of orbits of spaces invariant under unitary representations of discrete groups has been proved in \cite[Corollary 26]{BHP2018II}. The techniques introduced in the present setting differ from the ones used in \cite{BHP2018II}, due to the special nature of $\Gamma = \ZZ \rtimes G$ and the fact that $\ZZ$ is an abelian lattice of $\RR$.

\begin{theo}\label{theo:Parseval}
Let $\A = \{\varphi_i\}_{i = 1}^n\subset L^2(\RR)$ be a finite family. Then there exists a finite family with the same cardinality $\Psi = \{\psi_i\}_{i = 1}^n\subset L^2(\RR)$ such that
\begin{itemize}
\item[1.] $S_\Gamma(\A) = S_\Gamma(\Psi)$
\item[2.] $\{\tr{k}\rot{g}\psi_i \, : \, k \in \ZZ, g \in G, i \in \{1,\dots,n\}\}$ is a Parseval frame for $S_\Gamma(\A)$
\end{itemize}
\end{theo}
\begin{proof}
For a positive semidefinite Hermitian matrix $\G$, let us denote by $\G^+$ its Moore-Penrose pseudoinverse. Recall that
$$
\G^+ \G = \G \G^+  = \Proj_{\textnormal{Range}(\G)}.
$$
Since $\G^+$ commutes with $\G$, $(\G^+)^\frac12$ also commutes with $\G$, and hence
\begin{equation} \label{Eq:pseudoinverse}
(\G^+)^\frac12 \G (\G^+)^\frac12 = \G \G^+ = \Proj_{\textnormal{Range}(\G)}.
\end{equation}
For $\AG = \{\rot{g}\varphi_i : (i,g) \in I_n \times G\}$ define
$$
Q(\omega) := \TT_{\AG}(\omega)(\G_{\AG}(\omega)^+)^\frac12 \in \B(\C^{n \card},\ell_2(\ZZdual))\,,
$$
where $\TT_{\AG}$ and $\G_{\AG}$ are the pre-Gramian and Gramian, respectively, of $\AG$ (see Section \ref{Sec:fundamental} for the definition).
Then
\begin{align} \label{Eq:columns}
Q(\omega)^*Q(\omega) & = (\G_{\AG}(\omega)^+)^\frac12 \, \TT_{\AG}(\omega)^* \, \TT_{\AG}(\omega) \, (\G_{\AG}(\omega)^+)^\frac12\nonumber\\
& = (\G_{\AG}(\omega)^+)^\frac12 \, \G_{\AG}(\omega) \, (\G_{\AG}(\omega)^+)^\frac12 = \Proj_{\textnormal{Range}(\G_{\AG}(\omega))}.
\end{align}
For $j=1,2, \dots n\card\,,$ let $q_j(\omega) = Q(\omega)\delta_j$ be the columns of the matrix $Q(\omega)$, where $\{\delta_j\}_{j = 1,\dots,n\card}$ stands for the canonical basis of $\C^{n \card}.$ The vectors $q_j(\omega)$ belong to $\ell_2(\ZZdual)$ since for $c \in \C^{n \card}$

\begin{align*}
\int_\Omega \| Q(\omega) c\|_{\ell_2(\ZZdual)}^2 d\omega & =
\int_{\Omega }\langle Q(\omega) c, Q(\omega) c\rangle_{\ell_2(\ZZdual)}d\omega\\
& = \int_{\Omega }\langle c, \Proj_{\textnormal{Range}(\G_{\AG}(\omega))} c\rangle_{\ell_2(\ZZdual)}d\omega \leq 
 \|c\|^2 |\Omega| < \infty .
\end{align*}
Write  $\Theta(\omega) = \{q_j(\omega)\}_{j=1}^{n \card}.$ The Gram matrix of $\Theta(\omega)$ is $Q(\omega)^*Q(\omega)$. Therefore, by \eqref{Eq:columns} and \cite[Lemma 5.5.4]{Chr2003} (see also \cite[Corollary 7]{BHP2018}) we conclude that $\Theta(\omega)$ is a Parseval frame sequence in $\ell_2(\ZZdual).$
Moreover, the columns $\{q_j(\omega)\}_{j=1}^{n \card}$ of $Q(\omega)$ belong to $\J_{S_\Gamma(\A)}(\omega)$ because they are obtained as finite linear combinations (its coefficients are elements of $(\G_{\AG}(\omega)^+)^\frac12$) of the columns of $\TT_{\AG}(\omega)$, and these columns, which are $\T[\rot{g}\varphi_i](\omega)$, belong to $\J_{S_\Gamma(\A)}.$

Thus, by \cite[Theorem 4.1]{CP2010}, $\phi_j = \T^{-1}[q_j]$ are such that $\{\tr{k}\phi_j \, : \, k \in \ZZ, j = 1, \dots, n  \card \}$ form a Parseval frame sequence of $S_\Gamma(\A)$. Moreover, by \eqref{eq:preGramiancovariance}, \eqref{eq:Gramiancovariance}, and elementary functional calculus,
\begin{align*}
Q(g^*\omega) & = \TT_{\AG}(g^*\omega)(\G_{\AG}(g^*\omega)^+)^\frac12 = r_{g^{-1}}\TT_{\AG}(\omega)\,\lambda_g \, \lambda_{g^{-1}}(\G_{\AG}(\omega)^+)^\frac12 \,\lambda_g\\
& = r_{g^{-1}} Q(\omega)\, \lambda_g\,.
\end{align*}
Therefore, $Q$, which is the pre-Gramian of $\{\phi_j = \T^{-1}[q_j]\}_j^{n\card},$ satisfies (\ref{eq:preGramiancovarianceII}). By Theorem \ref{theo:preGramiancovariance}, there exist $\Psi = \{\psi_i\}_{i = 1}^n\subset L^2(\RR)$ such that $\rot{g}\psi_i = \phi_{i,g}$ for all $(i,g) \in I_n \times G$. Hence, $S_\Gamma(\A) = S_\Gamma(\Psi)$ and $\{\tr{k} \rot{g} \psi_i: k\in \ZZ, g\in G, \ i=1, \dots, n  \} = \{\tr{k}\phi_j \, : \, k \in \ZZ, j = 1, \dots, n \card \}$ form a Parseval frame of $S_\Gamma(\A).$
\end{proof}

\begin{cor}
Let $\A = \{\varphi_i\}_{i = 1}^n\subset L^2(\RR)$ be a finite family. Then there exists a finite family $\Psi = \{\psi_i\}_{i = 1}^\ell\subset L^2(\RR)$ with the same cardinality or less  such that
\begin{itemize}
\item[1.] $\{\tr{k} \rot{g} \psi_i: k\in \ZZ, g\in G\}$ is a Parseval frame of $S_\Gamma(\psi_i)$ for all $i=1,\dots, \ell$
\item[2.] $S_\Gamma(\psi_i) \, \bot \, S_\Gamma(\psi_j)$ for $i \neq j$
\item[3.] $S_\Gamma(\A) = \bigoplus_{i = 1}^\ell S_\Gamma(\psi_i)$.
\end{itemize}

\end{cor}
\begin{proof}
Let $\psi_1$ be the Parseval frame generator for $S_\Gamma(\varphi_1)$ constructed as in Theorem \ref{theo:Parseval}. Let then
$$
\phi_j = \Proj_{S_\Gamma(\varphi_1,\dots, \varphi_j)} \Proj_{S_\Gamma(\varphi_1)^\bot}\varphi_j, 2 \leq j \leq n.
$$
Let $n_2 = \min \{2\leq j\leq n: \phi_j \not=0\}$ and define  $\varphi_2' = \phi_{n_2}$.
Let $\psi_2$ again be the Parseval frame generator for $S_\Gamma(\varphi_2')$ constructed as in Theorem \ref{theo:Parseval}. 
Then
$$
S_\Gamma(\varphi_1,\dots, \varphi_{n_2}) = S_\Gamma(\varphi_1)\oplus S_\Gamma(\varphi_2') = S_\Gamma(\psi_1)\oplus S_\Gamma(\psi_2).
$$
Proceeding by induction, we get the whole family $\{\psi_i\}_{i = 1}^\ell$.
\end{proof}

\section{Borel sections and $\Gamma$-invariant range functions}\label{Sec:Section}
In this section we give conditions for the existence of Borel sections for the action of $\ZZdual \rtimes\, G$ on $\wh{\RR}$. When such a Borel section exists, we define a  $\Gamma$-invariant range function that characterizes $\Gamma$-invariant spaces.

\subsection{Borel sections}
If the group $G$ is finite, the present setting acquires an additional structure, that is discussed in this section and will be used in Section \ref{sec:approx}. We recall in passing that the finiteness of $G$, for discrete semidirect products $\Gamma = \ZZ \rtimes G$ with abelian normal component $\ZZ$, characterizes the type-I groups (see \cite[Theorem 7.8]{Folland2015}). Moreover, this assumption is satisfied by several relevant examples, most notably crystallographic groups, see Section \ref{Sec:example}. 

We first observe the following general fact involving those groups $\RR$ which are connected. Recall that, in this case, $\RR$ is isomorphic to $\R^n \times C$ for some $n \geq 0$ and $C$ a compact connected group (see e.g. \cite[Theorem 26]{Morris1977}).
\begin{prop}\label{prop:stabilizers}
Let $\RR$ be a second countable LCA group, let $\ZZ$ be a uniform lattice subgroup of $\RR$, let $G$ be a discrete and countable group of automorphisms of $\RR$ that preserves $\ZZ$, i.e. such that $g\ZZ = \ZZ$ for all $g \in G$ and let $\Gamma = \ZZ \rtimes G$.
If $\RR$ is connected, and if the action of $G$ on $\RR$ is faithful, then, for almost every $x \in \RR$, we have that 
$$
\textnormal{stab}_\Gamma(x) := \{(k,g) \in \ZZ \rtimes G : g x + k = x\} = \{(0,\id)\}.
$$
\end{prop}

\begin{proof}
Let $E = \{x \in \RR : \textnormal{stab}_\Gamma(x) \neq \{(0,\id)\}\}$. We will prove that $E$ is a set of measure zero. If, for $(k,g) \in \ZZ \rtimes G = \Gamma$, we denote by
$$
A(k,g) = \{x \in \RR : gx - x = k\},
$$
then we can write
\begin{align*}
E & = \Big\{x \in \RR : x = gx + k \textnormal{ for some } (k,g) \in \Gamma, (k,g) \neq (0,\id)\Big\}\\
& = \bigcup_{(k,g) \in \Gamma \smallsetminus \{(0,\id)\}} A(k,g) .
\end{align*}
Since this is a countable union, in order to prove $|E| = 0$, it suffices to prove that $|A(k,g)| = 0$ for all $(k,g) \in \Gamma \smallsetminus (0,\id)$. Suppose, by contradiction, that $|A(k,g)| > 0$ for some $(k,g) \neq (0,\id)$. Then (see e.g. \cite[Ch. 12, Sec. 61]{Halmos1950}) $A(k,g) - A(k,g)$ contains a neighborhood of the identity. Observe that $A(k,g) - A(k,g) \subset A(0,g)$, which is a subgroup of $\RR$, so the connectedness hypothesis implies (see e.g. \cite[Theorem 15]{Pontryagin1946}) that $A(0,g) = \RR$. Since $G$ acts faithfully on $\RR$, this is possible only for $g = \id$, hence providing a contradiction.
\end{proof}

Proposition \ref{prop:stabilizers} does not cover non-connected groups $\RR$. An example where every point of $\RR$ have nontrivial stabilizers is the following: let $\RR = \frac12 \Z$, let $\ZZ = \Z$ and let $G$ be multiplication by $\{1, -1\}$. Then, for all $x \in \frac12\Z$, the equation $x = gx + k$ for the stabilizer of $x$, with $g \in G$ and $k \in \Z$, is satisfied by the subgroup of $\Z \rtimes G$ whose elements are $\{(0,1), (2x,-1)\}$. More in general, it is easy to see that
if $G$ is nontrivial, then each $x \in \ZZ$ has a nontrivial stabilizer, so, if $\RR$ itself is discrete, then its Haar measure is the counting measure and the set of points with nontrivial stabilizers does not have zero measure.

We are now ready to establish the following fact, relating the existence of a Borel section with the finiteness of $G$.
\begin{prop}\label{prop:Effros}
Let $\RR$ be a connected second countable LCA group, let $\ZZ$ be a uniform lattice subgroup of $\RR$, let $G$ be a group that acts faithfully by automorphisms on $\RR$, and that preserves $\ZZ$, and let $\Gamma = \ZZ \rtimes G$. Then, the following are equivalent.
\begin{itemize}
\item[1.] $G$ is finite.
\item[2.] there exists a measurable $P \subset \RR$ of finite and positive $\RR$-Haar measure such that $\{\gamma P\}_{\gamma \in \Gamma}$ is an a.e. partition of $\RR$.
\end{itemize}
\end{prop}
\begin{proof}
Let us first assume $2$. Then, up to a zero measure set, we have
$$
\RR = \bigcup_{k \in \ZZ} \Big(\bigcup_{g \in G} gP \Big) + k.
$$
Thus, the set $Q := \displaystyle\bigcup_{g \in G} gP$ is a Borel section of $\RR/\ZZ$ of finite measure, and 
$$
|P|\, |G| = \sum_{g \in G} |gP| = |Q| < \infty.
$$
Hence, $G$ must be finite. Conversely, let us assume $1$.
Then, for almost all $x_0 \in \RR$, the orbit
$$
\mathcal{O}_{\Gamma}(x_0) = \{x \in \RR : x = \gamma x_0, \gamma \in \Gamma\}
$$
is closed, so the existence of a measurable set $P$ that intersects each orbit in exactly one point is ensured by \cite[Theorem 2.9]{Effros1965}. The fact that $P$ has positive measure is then due to the discreteness of $\Gamma$. Finally, by Proposition \ref{prop:stabilizers}, we have that $|\gamma P \cap \gamma' P| = 0$ when $\gamma \neq \gamma'$.
\end{proof}

\begin{rem}
For the group of translations and shears, see Section \ref{Sec:example}, there is no Borel section $P$, because the group $G$ of shears is not finite. Equivalently, one can see that the orbits can have accumulation points.
\end{rem}

If $G$ is finite, $\wh{\RR}$ is connected, and the action of $G$ on $\wh{\RR}$ is faithful, applying Proposition \ref{prop:Effros} to the semidirect product $\ZZdual \rtimes G$ acting on $\wh{\RR}$, one can deduce the existence of a measurable section $\Omega_0$ for the orbits space $\wh{\RR}/(\ZZdual \rtimes G)$. In particular, for such an $\Omega_0$ we have that the set $\Omega$ defined by
\begin{equation}\label{eq:Omegon}
\Omega := \displaystyle\bigcup_{g \in G} g^* \Omega_0
\end{equation}
is a Borel section of $\wh{\RR}/\ZZdual$. Moreover
\begin{equation}\label{eq:disjoint}
|g_1^*\Omega_0 \cap g_2^*\Omega_0| = 0 \quad \mbox{for all} \quad g_1, g_2 \in G, g_1 \neq g_2 .
\end{equation}

\subsection{$\Gamma$-invariant range functions}

When a Borel section $\Omega_0$  exists, the $\Gamma$-inva\-riant spaces can be characterized by special range functions, similarly as in the case of shift-invariant spaces.

To see this, let us  define the map 
$\TG: L^2(\RR) \longrightarrow  L^2(\Omega_0 \times G,\ell_2(\ZZdual))$ by 
\begin{equation}\label{def:mappingTG}
\TG[f] (\omega, g)  = \Big\{\wh{f}(g^*\omega + k)\Big\}_{k \in \ZZdual} = \T[f](g^*\omega) \quad \textnormal{a.e.} \ \omega \in {\Omega_0}.
\end{equation}
In a similar way as with the $\T$ map given in Definition \ref{def:mappingT} one can see that $\TG$ is an isometric isomorphism.

A direct calculation using Lemma \ref{lem:intertwining} shows that for $u \in G$ and $s \in \ZZdual$
\begin{align*}
\TG[\rot{u}f](\omega, g) & =  \Big\{\wh{f}(u^*g^*\omega +u^*k)\Big\}_{k\in \ZZdual} = r_u (\TG[f](\omega,gu) )\\
\TG[\tr{s}f](\omega, g) & = e^{-2\pi i g^*\omega. s}\,\TG[f](\omega,g).
\end{align*}
This justifies the following definition.
\begin{defi}\label{Gamma-invariant-def}
Let $\Omega_0 \subset \wh{\RR}$ be a Borel section of $\wh{\RR}/(\ZZdual \rtimes G).$ A measurable $\Gamma$-invariant range function is a measurable map
$$
\K : \Omega_0 \times G \longrightarrow \{\textnormal{closed subspaces of } \ell_2(\ZZdual)\}
$$
such that the closed subspace $\K(\omega,g)$ satisfies the extra condition that
\begin{equation}\label{eq-inv}
r_{u^{-1}}(\K(\omega,g)) = \K(\omega,gu) \textnormal{\ for a.e. $\omega\in \Omega_0$, for all $g, u\in G.$}
\end{equation}
\end{defi}
We say that $\K$ is measurable if the family $\pj_{\K(\omega,g)} \in \B(\ell_2(\ZZdual))$ of orthogonal projections onto $\K(\omega,g)$ is measurable. 
These measurable $\Gamma$-invariant range functions share most of the properties of measurable range functions of Definition \ref{rf-def}. Their proofs can be obtained by a proper adaptation of the classical case \cite{CP2010}. The differences are due to the fact that $\Gamma$-invariant range functions come from the action of a (non-commutative) semidirect product instead of just translations by a commutative group.

Given a $\Gamma$-invariant range function $\K$, denote by 
$$
M_\K = \{F \in L^2(\Omega_0 \times G,\ell_2(\ZZdual)) : F(\omega,g) \in \K(\omega,g) \textnormal{ a.e. } \omega \in \Omega_0, \forall g \in G\}.
$$
Since $\K(\omega,g)$ is closed for every $g \in G$, and a.e. $\omega \in \Omega_0$, it is easy to see that  $M_\K$ is closed. With this notation, we have the following lemmas, whose proofs can be obtained following the arguments of \cite{CP2010}.

\begin{lem} \label{Lem:Gamma-range}
Let $\K$ be a measurable $\Gamma$-invariant range function. Denote by $\Proj_{M_\K} \in \B(L^2(\Omega_0 \times G,\ell_2(\ZZdual)))$ the orthogonal projection onto $M_\K$, and for $(\omega,g) \in \Omega_0\times G$ let $\pj_{\K(\omega,g)} \in \B(\ell_2(\ZZdual))$ be the orthogonal projection onto $\K(\omega,g)$. Then
$$
(\Proj_{M_\K}F)(\omega,g) = \pj_{\K(\omega,g)}(F(\omega,g)) \, , \ \forall \, F \in L^2(\Omega_0 \times G,\ell_2(\ZZdual)) \, , \ \textnormal{a.e. } \omega \in \Omega_0, \forall \ g \in G .
$$
\end{lem}

\begin{lem}\label{lem:uniquenessK}
Let $\K, \K'$ be two measurable $\Gamma$-invariant range functions. Then $M_\K = M_{\K'}$ if and only if $\K(\omega,g) = \K'(\omega,g)$ for a.e. $\omega \in \Omega_0$, for all $g \in G$.
\end{lem}

The next theorem characterizes $\Gamma$-invariant spaces in terms of $\Gamma$-invariant range functions (cf. Theorem \ref{Th:Helson}).
\begin{theo}\label{Th:Helson-2}
Assume that a Borel section $\Omega_0$ for the action of \, $\wh{\Gamma}\rtimes G$ on $\wh{\RR}$ exists, let $\V$ be a closed subspace of $L^2(\RR)$ and let $\TG$ the map (\ref{def:mappingTG}). The subspace $\V$ is $\Gamma$-invariant if and only if there exists a unique measurable $\Gamma$-invariant range function $\K_\V$ such that 
\begin{equation}\label{cc}
\V=\{ f\in L^2(\RR): \TG[f](\omega,g) \in \K_{\V}(\omega,g),\  g \in G,\  a.e.\; \omega \in \Omega_0)\}.
\end{equation}
Moreover, if $\V=S_{\Gamma}(\A)$ for some countable set $\A$ of $L^2(\RR)$, the measurable $\Gamma$-invariant range function associated to $S_{\Gamma}(\A)$ satisfies for each $u \in G$ that
$$
\K_\V(\omega,u) =\ol{\vsp}\{\TG[R_g\varphi](\omega,u)\, : \, \varphi \in \A, g \in G\}, \quad a.e. \ \omega \in \Omega_0, \forall \ u \in G.
$$
\end{theo}

\begin{proof} 
Assume first that $\V$ is $\Gamma$-invariant. We want to prove the existence of a measurable $\Gamma$-invariant range function $\K$ associated to $\V.$
Let $\A\subset L^2(\RR)$ be a countable set such that $\V=S_{\Gamma}({\A})$. Define
\begin{equation}\label{dd}
\K(\omega,u) = \ol{\vsp}\{\TG[R_g\varphi](\omega,u) \, : \, \varphi \in \A, g \in G\}, \quad a.e. \ \omega \in \Omega_0, \forall \ u \in G.
\end{equation}
It is straightforward to see that $\K$ is a $\Gamma$-invariant range function.

Further it is clear that equation \eqref{cc} is true if and only if $\TG(\V)=M_{\K}.$ 
Set now  $M:=\TG(\V).$
We will show that $M =  M_\K. $

To see that $M \subset M_\K$, let $F \in M$ and for each $n \in \N$ choose $f_n \in$ span$\{R_g\varphi:\varphi \in \A, g \in G\}\subset \V$ such that $F_n:=\TG(f_n) \rightarrow F$ in $M$. So, $F_{n_k}(\omega,g)$ converges to $F(\omega,g)\; a.e.\, \omega \in \Omega_0, \forall \ g \in G$, for some subsequence $\{F_{n_k}\}_k.$
Since $F_{n_k}(\omega,g) \in \K(\omega,g)$ for all $k$ and for a.e. $\omega \in \Omega_0$ and all $g \in G$, the fact that $\K(\omega,g)$ is closed implies that $F(\omega,g)\in \K(\omega,g)$ for a.e. $\omega \in \Omega_0$ and all $g \in G$. 

To see that $M$ cannot be a proper subspace of $M_\K$, assume that  $F \in M_\K$, and $F \bot M$. Then, for all $\phi \in M$ and all $k \in \ZZ$, denoting by $f = \TG^{-1}[F]$, and by $\varphi = \TG^{-1}[\phi]$, we have that $\TG[\tr{k}\varphi] \in M$. Since
$$
\TG[\tr{k}\varphi](\omega,g) = e^{-2\pi i g^*\omega.k}\,\TG[\varphi](\omega,g)
$$
we have
\begin{align*}
0 & = \langle F, \TG[\tr{k}\varphi] \rangle = \int_{\Omega_0} \sum_{g \in G}\langle F(\omega,g), \TG[\tr{k}\varphi](\omega,g)\rangle_{\ell_2(\ZZdual)} d\omega\\
& = \int_{\Omega_0} \sum_{g \in G} e^{-2\pi i g^*\omega. k} \sum_{\ell \in \ZZdual} \wh{f}(g^*\omega + \ell) \ol{\wh{\varphi}(g^*\omega + \ell)} d\omega\\
& = \int_{\Omega} e^{-2\pi i \omega . k} \sum_{\ell \in \ZZdual} \wh{f}(\omega + \ell) \ol{\wh{\varphi}(\omega + \ell)} d\omega\\
& = \int_{\Omega} e^{-2\pi i \omega. k} \langle \T[f](\omega), \T[\varphi] (\omega)\rangle_{\ell_2(\ZZdual)} d\omega .
\end{align*}
Using that the last integral is a Fourier coefficient with respect to the character $e^{-2\pi i \omega . k}$, we have that $\langle \T[f](\omega), \T[\varphi] (\omega)\rangle_{\ell_2(\ZZdual)} = 0$ for a.e. $\omega \in \Omega$. Thus, by definition of $\TG$, we have
$$
\langle F(\omega,g), \phi(\omega,g)\rangle_{\ell_2(\ZZdual)} = \langle \T[f](g^*\omega),\T[\varphi](g^*\omega)\rangle_{\ell_2(\ZZdual)} = 0 \quad \textnormal{a.e. } \omega \in \Omega_0 .
$$
Since $\TG[\rot{g}\varphi] \in M$ for all $g \in G$ and all $\varphi \in \Phi$, we then have
$$
\langle F(\omega,u), \TG[\rot{g}\varphi](\omega,u)\rangle_{\ell_2(\ZZdual)} = 0.
$$
By (\ref{dd}), $F(\omega,g)\in \K(\omega,g)^{\perp}$ for a.e. $\omega \in \Omega_0$ and all $g \in G$. Since  $F(\omega,g)\in \K(\omega,g)$, we conclude that $F = 0$.

Regarding the measurability of $\K$, it is enough to note that for $g = \id$, $\K(\omega,\id)$ is the restriction of the range function $\T$ for the shift invariant space $\V=S(R_g\varphi: g\in G, \varphi \in \Phi)$ to the measurable subset $\Omega_0$ and hence measurable.
For any other $g \in G$, $\K(\omega,g)= r_g^{-1} (\K(\omega,\id))$ and therefore also measurable.

So, we  proved that for each $\Gamma$-invariant subspace $\V$ there exists a measurable $\Gamma$-invariant range function, say $\K_{\V}$, that satisfies \eqref{cc}. This function is unique due to Lemma \ref{lem:uniquenessK}.

Conversely, assume that $\K$ is a measurable $\Gamma$-invariant range function. Define $\V$ by \eqref{cc}. Then, $\V$ is $\Gamma$-invariant, for if $f\in \V$, $\TG[f](\omega,g) \in \K(\omega,g)\; a.e. \,\omega$ and, $\TG[R_uf](\omega,g) = r_u \TG[f](\omega,gu) \in r_u \K(\omega,gu) = \K(\omega,g)$ by the $\Gamma$-invariance of $\K$.
\end{proof}

\section{Approximation}\label{sec:approx}

In this section we study the approximation problem mentioned in the introduction.
The idea is to find a low dimensional model (a subspace), among a certain class of subspaces, that best fits a given dataset.
The subspace will be optimal for the data in the sense that it minimises a proposed functional.
This problem has been studied in the case the class of subspaces are shift invariant spaces \cite{ACHM2007, CM2016}.
The importance of our approach is that our class includes subspaces that are invariant by rigid movements in $\R^d$.
This is very important in applications since we are able to include rotations and symmetries.

Throughout this section, we will always assume that $G$ is finite, and that a Borel section of $\wh{\RR}/(\ZZdual \rtimes G)$ exists. Recall that, by Proposition \ref{prop:Effros}, such a section always exists whenever $G$ acts faithfully on $\wh{\RR}$ and $\wh{\RR}$ is connected.

The next theorem is the main result of this section, and it shows that for any given dataset, and a positive number $\kappa$, there always exists an optimal $\Gamma$-invariant subspace of length at most $\kappa$.
We also obtain the exact error of approximation and a formula for a Parseval set of generators.
 
Let $\data = \{f_1,\dots,f_m\} \subset L^2(\RR)$ and let $\V \subset L^2(\RR)$ be a closed subspace. Define the functional 
\begin{equation}\label{eq:functional}
\E[\V; \data] = \sum_{i=1}^m \|f_i - \Proj_{\V}f_i\|^2_{L^2(\RR)}.
\end{equation}
Note that the value of the functional $\E$ in $\V$ and $\data$ is the sum of the squares of the distances between each element of $\data$ and the closed subspace $\V$, which will be chosen in the class of $\Gamma$-invariant subspaces $S_\Gamma(\A)$ introduced in Section \ref{Sec:structure}.

Before stating the main theorem concerning the minimization of (\ref{eq:functional}), we need the following properties of eigenvalues and eigenvectors of the Gramian of the data.

\begin{lem}\label{lem:eigenG}
Let $\DG$ be the family $\{R(g)f_i : (i,g) \in I_m \times G\} \subset L^2(\RR)$ ordered with the lexicographical ordering of $I_m \times G := \{1, 2, \dots, m\} \times G$, and let $\G := \G_{\DG}$ be its Gramian as in Definition \ref{def:gramian}.
\begin{itemize}
\item[1.] For $\omega \in \Omega$, let $\{\sigma_{i,g}(\omega)^2 : (i,g) \in I_m \times G\}$ be the eigenvalues of $\G(\omega)$  ordered decreasingly with the lexicographical ordering of $I_m \times G$, counted with their multiplicity. Then they are G-invariant, in the sense that
\begin{equation}\label{eq:eigenvaluesinvariance}
\sigma_{i,g}(g_0^*\omega) = \sigma_{i,g}(\omega) \quad \forall \ (i,g) \in I_m \times G, \ \forall \, g_0 \in G, \textnormal{a.e. } \omega \in \Omega. 
\end{equation}
\item[2.] For $\omega \in \Omega_0$, let $\{V^{i,g}(\omega) : (i,g) \in I_m \times G\} \subset \C^{m\card}$ be the corresponding orthonormal eigenvectors of $\G(\omega)$, and denote the components of the $(i,g)$-th eigenvector by $\{V^{i,g}_{j,q}(\omega) : (j,q) \in I_m \times G\} \subset \C$. Then, it is possible to obtain a family of orthonormal eigenvectors of $\G(\omega)$ at a.e. $\omega \in \Omega$ whose components satisfy
\begin{equation}\label{eq:eigenvectorscovariance}
V^{i,g}_{j,q}(g_0^*\omega) = V^{i,g}_{j,g_0q}(\omega) \quad \forall \ g_0 \in G, \, \textnormal{a.e. } \omega \in \Omega.
\end{equation}
\end{itemize}
\end{lem}
\begin{proof}
To prove 1. observe that, by Corollary \ref{cor:Gramcovariance}, the Gramian $\G$ of $\DG$ satisfies (\ref{eq:Gramiancovariance}), so the spectrum of $\G(g^*\omega)$ coincides with that of $\G(\omega)$ at a.e. $\omega \in \Omega$, for all $g \in G$, because $\lambda_{g}$ is unitary. Then, (\ref{eq:eigenvaluesinvariance}) holds due to the decreasing ordering imposed at a.e. $\omega \in \Omega$. To prove 2, for a.e. $\omega \in \Omega$ let $\omega_0 \in \Omega_0$ and $g_0 \in G$ be the unique elements such that $\omega = g_0^*\omega_0$, and for $(i,g) \in I_m \times G$ let $z \in \C^{m\card}$ satisfy
$$
\G(\omega) z = \sigma_{i,g}(\omega) z.
$$
Using (\ref{eq:Gramiancovariance}) and (\ref{eq:eigenvaluesinvariance}) we deduce
$$
\G(\omega_0) \lambda_{g_0} z = \sigma_{i,g}(\omega_0) \lambda_{g_0} z ,
$$
so we can take $V^{i,g}$ at $\omega$ in such a way that
\begin{equation}\label{eq:defineeigenvectorsG}
\lambda_{g_0} V^{i,g}(\omega) = V^{i,g}(\omega_0). 
\end{equation}
Thus, by definition of $\lambda$
$$
V^{i,g}_{j,q}(g_0^*\omega_0) = (\lambda_{g_0^{-1}} V^{i,g}(\omega_0))_{j,q} = V^{i,g}_{j,g_0q}(\omega_0) \quad \textnormal{a.e. } \omega_0 \in \Omega_0.
$$
The validity of this identity for a.e. $\omega \in \Omega$ follows from this, using (\ref{eq:defineeigenvectorsG}) and the definition   of $\lambda$.
\end{proof}

\begin{theo}\label{theo:main}
Let $\data = \{f_1,\dots,f_m\}$ be a set of functional data in $L^2(\RR)$. Using the same notations as in Lemma \ref{lem:eigenG}, the following holds:

\vspace{1ex}
\noindent
\textnormal{1.} For all $\kappa \in \{1,\dots,m\}$ there exists a $\Gamma$-invariant space $\W \subset L^2(\RR)$ generated by $\Gamma$-orbits of a family $\{\psi_i\}_{i=1}^\kappa \subset L^2(\RR)$ such that
$$
\E[\W ;\data] = \min \{\E[\V;\data] : \V  \subset L^2(\RR) \, ,\Gamma \text{-invariant and }  \len(\V) \leq \kappa \}
$$
and the system $\{\tr{k}\rot{g}\psi_i \, : \, k \in \ZZ, g \in G, i\in\{1,\dots,\kappa\}\}$ is a Parseval frame of $\W$.

\vspace{1ex}
\noindent
\textnormal{2.} The approximation error for the minimizing space $\W$ is given by
$$
\E[\W;\data] = \sum_{i = \kappa + 1}^m \sum_{g \in G} \int_{\Omega_0} \sigma_{(i,g)}(\omega)^2 d\omega .
$$

\vspace{1ex}
\noindent
\textnormal{3.} A family $\{\psi_i\}_{i=1}^\kappa \subset L^2(\RR)$ that generates a minimizer $\W$ is given by
\begin{equation}\label{eq:generators}
\T[\psi_i](\omega) = \sum_{(j,g') \in I_m\times G} C_i^{j,g'}(\omega) \T[\rot{g'}f_j](\omega)
\end{equation}
where
\begin{equation}\label{eq:coefficients}
C_i^{j,g'}(\omega) = \sum_{g \in G} \theta_{i,g}(\omega) V_{j,g'}^{i,g} (\omega) \one{g^*\Omega_0}(\omega) \ , \quad i=1,\dots,\kappa
\end{equation}
and $\theta_{i,g}(\omega) = (\sigma_{i,g}(\omega))^{-1}$ if $\sigma_{i,g}(\omega) \neq 0$ and $0$ otherwise. All identities hold for a.e. $\omega \in \Omega$.
\end{theo}

The proof of Theorem \ref{theo:main} relies on the minimization of a family of related functionals, which read as follows.

\begin{defi}\label{def:minimizeromega}
Let $\V \subset L^2(\RR)$ be a $\ZZ$-invariant subspace, $\data = \{f_1,\dots,f_m\} \subset L^2(\RR)$ and $\omega \in \wh{\RR}$. We define
$$
\mathcal D (\V; \data)(\omega) =
\sum_{i=1}^m \sum_{g \in G} \|\T[\rot{g}f_i](\omega) - \pj_{\J_\V(\omega)}\T[\rot{g}f_i](\omega)\|^2_{\ell_2(\ZZdual)} .
$$
For $\omega \in \wh{\RR}$, we say that a $\Gamma$-invariant subspace $\W \subset L^2(\RR)$ is a minimizer of length at most $\kappa$ of $\mathcal D (\V; \data)(\omega)$ if $\len(\W)\leq \kappa$, and
$$
\mathcal D (\W; \data)(\omega) \leq \mathcal D (\V; \data)(\omega) \quad \forall \ \V \ \Gamma\textnormal{-invariant} \, : \len(\V)\leq \kappa .
$$
\end{defi}

We have the following result relating $\E[\V; \data]$ and $\mathcal D (\V; \data).$

\begin{prop}\label{prop:errors}
Let $\V \subset L^2(\RR)$ be $\Gamma$-invariant subspace. Then it holds,
\begin{equation} \label{Pro:minimize}
\E[\V; \data] = \int_{\Omega_0} \mathcal D (\V; \data)(\omega) \, d\omega\,.
\end{equation}
Moreover, if $\W \subset L^2(\RR)$ is a minimizer of length at most $\kappa$ of $\mathcal D (\V; \data)(\omega)$ for almost all $\omega \in \Omega_0,$ then $\W$ is also a minimizer of $\E[\V; \data]$ over all $\V$ with $\len(\V) \leq \kappa.$        
\end{prop}

\begin{proof}
To prove \eqref{Pro:minimize}, use that $\T$ is an isometry from $L^2(\RR)$ onto $L^2(\Omega,\ell_2(\ZZdual))$ (see Lemma \ref{lem:gram}) to write
$$
\E[\V;\data]  = \sum_{i=1}^m \|\T[f_i] - \T[\Proj_{\V}f_i]\|^2_{L^2(\Omega,\ell_2(\ZZdual))}\,.
$$
Furthermore it is easy to see that $\T \Proj_{\V} = \Proj_{\T[\V]}\T$, so we have
$$
\E[\V;\data] = \sum_{i=1}^m \|\T[f_i] - \Proj_{\T[\V]}\T[f_i]\|^2_{L^2(\Omega,\ell_2(\ZZdual))} \,.
$$
Since $\V$ is $\Gamma$-invariant, hence in particular it is $\Lambda$-invariant, by Theorem \ref{Th:Helson} we have $\T[\V] = M_{\J_\V}$. Hence,
$$
\E[\V;\data]  = \sum_{i=1}^m \|\T[f_i] - \Proj_{M_{\J_\V}}\T[f_i]\|^2_{L^2(\Omega,\ell_2(\ZZdual))}\,.
$$
By definition of the norm in $L^2(\Omega,\ell_2(\ZZdual))$, and using Lemma \ref{Lem:range}, we can write
\begin{align*}
\E[\V;\data] &
= \sum_{i=1}^m \int_{\Omega} \|\T[f_i](\omega) - (\Proj_{M_{\J_\V}}\T[f_i])(\omega)\|^2_{\ell_2(\ZZdual)}d\omega\\
& =  \sum_{i=1}^m \int_{\Omega} \|\T[f_i](\omega) - \pj_{\J_\V(\omega)}\T[f_i](\omega)\|^2_{\ell_2(\ZZdual)}d\omega\\
& =  \sum_{i=1}^m \sum_{g \in G} \int_{g^*\Omega_0} \|\T[f_i](\omega) - \pj_{\J_\V(\omega)}\T[f_i](\omega)\|^2_{\ell_2(\ZZdual)}d\omega \,.
\end{align*}
An obvious change of variables gives
\begin{align*}
\E[\V;\data] & =  \sum_{i=1}^m \sum_{g \in G} \int_{\Omega_0} \|\T[f_i](g^*\omega) - \pj_{\J_\V(g^*\omega)}\T[f_i](g^*\omega)\|^2_{\ell_2(\ZZdual)}d\omega\\
& = \int_{\Omega_0} \sum_{i=1}^m \sum_{g \in G} \|\T[f_i](g^*\omega) - \pj_{\J_\V(g^*\omega)}\T[f_i](g^*\omega)\|^2_{\ell_2(\ZZdual)}d\omega\,.
\end{align*}
By Lemma \ref{lem:intertwining} and Theorem \ref{theo:rangecovariance},
$$
\E[\V;\data] =   \int_{\Omega_0} \sum_{i=1}^m \sum_{g \in G} \|r_{g^{-1}}\T[\rot{g}f_i](\omega) - \pj_{r_{g^{-1}}\J_\V(\omega)}\,r_{g^{-1}}\T[\rot{g}f_i](\omega)\|^2_{\ell_2(\ZZdual)}d\omega \,.
$$
Now, using that $r_g$ is an isometry, which implies $\pj_{r_{g^{-1}}\J_\V(\omega)} = r_{g^{-1}}\pj_{\J_\V(\omega)} r_{g}$, we can conclude
\begin{align*}
\E[\V;\data] & = \int_{\Omega_0} \sum_{i=1}^m \sum_{g \in G} \|r_{g^{-1}}\T[\rot{g}f_i](\omega) - r_{g^{-1}}\pj_{\J_\V(\omega)}\T[\rot{g}f_i](\omega)\|^2_{\ell_2(\ZZdual)}d\omega \\
& = \int_{\Omega_0} \mathcal D (\V; \data)(\omega)\, d\omega .
\end{align*}
This proves \eqref{Pro:minimize}.

To conclude the proof of the proposition, suppose that $\W$ is a minimizer of $\mathcal D (\V;\data)(\omega)$ of lenght at most $\kappa$ for almost all $\omega \in \Omega_0.$ Then, for any $\V \subset L^2(\RR)$ with $\len(\V) \leq \kappa$, we have
\[
\E[\W;\data] = \int_{\Omega_0} \mathcal D (\W; \data)(\omega)\, d\omega \leq \int_{\Omega_0} \mathcal D (\V; \data)(\omega)\, d\omega = \E[\V;\data]. \qedhere
\]
\end{proof}

\subsection*{Proof of Theorem \ref{theo:main}}
The proof relies on similar arguments to the ones provided in \cite{ACHM2007}.

Let us first prove $1.$
Let $\omega \in \Omega$, $I_m=\{1,\dots,m\}$, and $F(\omega)$ be the pre-Gramian of the system $\DG$, evaluated in $\omega$, that is,
$$
F(\omega) c = \sum_{(i,g) \in I_m\times G} c_{i,g} \,a_{i,g}(\omega), \quad c \in \C^{m \card}
$$
where
$$
a_{i,g}(\omega) = \T[\rot{g}f_i](\omega) \in \ell_2(\ZZdual) \quad i \in I_m ,\; g \in G .
$$

For each $\omega \in \Omega_0$, consider $\G(\omega) = F(\omega)^*F(\omega)$ the Gramian of $\DG$ at $\omega$, and let $\{\sigma_{(i,g)}(\omega)^2 , (i,g) \in I_m \times G\}$ be its eigenvalues, ordered decreasingly with the lexicographical ordering of $I_m \times G$. Then $\G(\omega) = V(\omega) \Sigma(\omega)^2 V(\omega)^*$ where $V(\omega)$ is the square matrix of order $m\card$ whose columns are the orthonormal eigenvectors $\{V^{i,g}(\omega) : (i,g) \in I_m \times G\} \subset \C^{m\card}$ of $\G(\omega)$, and $\Sigma(\omega) = \textnormal{diag}(\sigma_{i,g}(\omega))$. 
Define
$$
U(\omega) = F(\omega) V(\omega) \Sigma(\omega)^+ ,
$$
where $\Sigma(\omega)^+ = \textnormal{diag}(\theta_{i,g}(\omega))$, and $\theta_{i,g}(\omega) = (\sigma_{i,g}(\omega))^{-1}$ when $\sigma_{i,g}(\omega) \neq 0$, and $0$ otherwise. Then
\begin{equation}\label{eq:UParseval}
U(\omega)^*U(\omega) = \Sigma(\omega)^+ V(\omega)^* F(\omega)^* F(\omega) V(\omega) \Sigma(\omega)^+ = \Sigma(\omega)^+ \Sigma(\omega)^2 \Sigma(\omega)^+
\end{equation}
is a diagonal matrix with $1$ on the first $r$ entries, and $0$ on the other ones, where $r = \textnormal{rank}(F(\omega))$.

Let $u_{i,g}(\omega)$ be the $(i,g)$-th column of $U(\omega)$, and for $\kappa \in I_m$ and $\omega \in \Omega_0$ denote by $\wt{U}(\omega)$ the matrix whose columns are given by the family $\{u_{i,g}(\omega)\,:\,i=1,\dots,\kappa, g \in G\} \subset \ell_2(\ZZdual)$. By (\ref{eq:Omegon}) and (\ref{eq:disjoint}), we can extend $\wt{U}$ to the whole $\Omega$ by defining
\begin{equation}\label{eq:extensionH}
H(g^*\omega) = r_{g^{-1}} \wt{U}(\omega) \lambda_g  \ , \quad \omega \in \Omega_0 , \ g \in G.
\end{equation}

Let $\{h_{i.g}(\omega) \, , \ (i,g) \in I_\kappa \times G\}$ be the columns of $H(\omega)$. Denoting by $\{\delta_{i,g}\, , \ (i,g) \in I_\kappa \times G\} \subset \C^{\kappa\card}$ the canonical basis of $\C^{\kappa\card}$, we have
$$
h_{i.g'}(g^*\omega) = H(g^*\omega)\delta_{i,g'} = r_{g^{-1}} \wt{U}(\omega) \lambda_g \delta_{i,g'} = r_{g^{-1}} \wt{U}(\omega) \delta_{i,gg'} = r_{g^{-1}} u_{i,gg'}(\omega)
$$
for $\omega \in \Omega_0$. Since, for a.e. $\omega \in \Omega$, there exists a unique $g \in G$ and a unique $\omega_0 \in \Omega_0$ such that $\omega = g^*\omega_0$, we have
$$
h_{i.g'}(\omega) = h_{i.g'}(g^*\omega_0) = r_{g^{-1}} u_{i,gg'}((g^{-1})^*\omega).
$$
Therefore we can then write, for any $\omega \in \Omega$
\begin{equation}\label{eq:extensionh}
h_{i,g'} (\omega) = \sum_{g \in G} r_{g^{-1}} u_{i,gg'}((g^{-1})^*\omega) \one{g^*\Omega_0}(\omega) .
\end{equation}
Using (\ref{eq:UParseval}), we can see that $h_{i,g} \in L^2(\Omega,\ell_2(\ZZdual))$ for all $(i,g) \in I_\kappa \times G$, because
\begin{align*}
\int_\Omega & \|h_{i,g'} (\omega)\|_{\ell_2(\ZZdual)}^2 d\omega = \sum_{g \in G} \int_{g^*\Omega_0} \|r_{g^{-1}} u_{i,gg'}((g^{-1})^*\omega)\|_{\ell_2(\ZZdual)}^2 d\omega\\
& = \sum_{g \in G} \int_{\Omega_0} \|u_{i,g}(\omega)\|_{\ell_2(\ZZdual)}^2 d\omega = \sum_{g \in G} \int_{\Omega_0} \|U(\omega) \delta_{i,g}\|_{\ell_2(\ZZdual)}^2 d\omega\\
& = \sum_{g \in G} \int_{\Omega_0} \langle U(\omega)^*U(\omega) \delta_{i,g}, \delta_{i,g}\rangle_{\ell_2(\ZZdual)} d\omega \leq \card|\Omega_0| = |\Omega|.
\end{align*}

Thus, $H$ is the pre-Gramian of the family $\{\psi_{i,g} \, : \, (i,g) \in I_\kappa \times G\} \subset L^2(\RR)$ defined by
\begin{equation}\label{eq:rotatedgen}
\psi_{i,g} := \T^{-1}[h_{i,g}].
\end{equation}
Moreover, for a.e. $\omega \in \Omega$, let $g_0 \in G$ and $\omega_0 \in \Omega_0$ be the unique elements satisfying $\omega = g_0^*\omega_0$. Then, for all $g \in G$, by (\ref{eq:extensionH}) we have
\begin{align*}
H(g^*\omega) & = H((g_0g)^*\omega_0) = r_{(g_0g)^{-1}} \wt{U}(\omega_0) \lambda_{(g_0g)} = r_{g^{-1}} r_{g_0^{-1}}\wt{U}(\omega_0) \lambda_{g_0}\lambda_g\\
& = r_{g^{-1}} H(g_0^*\omega_0) \lambda_g = r_{g^{-1}} H(\omega) \lambda_g .
\end{align*}
By the second part of Theorem \ref{theo:preGramiancovariance}, this implies that for all $(i,g) \in I_\kappa \times G$ we have
\begin{equation}\label{eq:proofgenerators}
\psi_{i,g} = \rot{g}\psi_{i,\id}.
\end{equation}
We now show that the family $\{\psi_i = \psi_{i,\id}\}_{i = 1}^\kappa$ generates, under the action of $\Gamma$, a minimizing subspace. Let
$$
\W = \ol{\vsp}\{\tr{k}\rot{g}\psi_i \, : \ i \in I_\kappa, k \in \ZZ, g \in G\} \subset L^2(\RR).
$$
Clearly $\W$ is $\Gamma$-invariant, and $\len(\W) \leq \kappa$. By Theorem \ref{Th:Helson}, and using (\ref{eq:rotatedgen}) and (\ref{eq:proofgenerators}), the range function of $\W$ (as a $\ZZ$-invariant space) is
$$
\J_\W(\omega) = \vsp\{\T[\rot{g}\psi_i](\omega) : (i,g) \in I_\kappa \times G\} = \vsp\{h_{i,g}(\omega) : (i,g) \in I_\kappa \times G\}
$$
for a.e. $\omega \in \Omega$. In particular,
$$
\J_\W(\omega) = \vsp\{u_{i,g}(\omega) : (i,g) \in I_\kappa \times G\} \quad \textrm{a.e. } \omega \in \Omega_0 .
$$

Now, for a.e. $\omega \in \Omega_0$, by applying the Schmidt-Eckart-Young theorem \cite{Sch1907, EY1936} to the data $\{a_{i,g}(\omega) \, , \ (i,g) \in I_m \times G\}$ (see also \cite[Chapt. VI.1, Theorem 1.5]{GGK1990} and \cite[Theorem 4.1]{ACHM2007}), we get that the subspace $\J_\W(\omega) \subset \ell_2(\ZZdual)$ is a minimizer of length at most $\kappa$ of $\mathcal D (\V; \data)(\omega)$ in the sense of Definition \ref{def:minimizeromega}. Thus, by Proposition \ref{prop:errors}, we have that $\W$ is a minimizer for $\E[\V;\data]$ over all $\Gamma$-invariant subspaces $\V \subset L^2(\RR)$ with $\len(\V) \leq \kappa$.

It remains to show that $\{\tr{k}\rot{g}\psi_i \, : \ i \in I_\kappa, k \in \ZZ, g \in G\}$ is a Parseval frame of $\W$. By (\ref{eq:UParseval}) together with \cite[Lemma 5.5.4]{Chr2003} (see also \cite[Corollary 7]{BHP2018}), the system $\{u_{i,g}(\omega) : (i,g) \in I_\kappa \times G\}$ is a Parseval frame of $\J_\W(\omega)$ for a.e. $\omega \in \Omega_0$. We now prove that $\{h_{i,g}(\omega) : (i,g) \in I_\kappa \times G\}$ is a Parseval frame of $\J_\W(\omega)$ for a.e. $\omega \in \Omega$. Let $\omega_0 \in \Omega_0$ and $g_0 \in G$ be the unique elements satisfying $\omega = g_0^*\omega_0$, and let $c \in \J_\W(\omega)$. Using (\ref{eq:extensionh})
\begin{align*}
\sum_{i = 1}^\kappa \sum_{g \in G} |\langle h_{i,g}(\omega),c\rangle_{\ell_2(\ZZdual)}|^2 & = \sum_{i = 1}^\kappa \sum_{g \in G} |\langle r_{g_0^{-1}}u_{i,g_0g}(\omega_0),c\rangle_{\ell_2(\ZZdual)}|^2\\
& = \sum_{i = 1}^\kappa \sum_{g \in G} |\langle u_{i,g_0g}(\omega_0),r_{g_0}c\rangle_{\ell_2(\ZZdual)}|^2 .
\end{align*}
By Theorem \ref{theo:rangecovariance}, since $c \in \J_\W(\omega) = \J_\W(g_0^*\omega_0)$, then $r_{g_0}c \in \J_\W(\omega_0)$. Since $\{u_{i,g}(\omega_0) : (i,g) \in I_\kappa \times G\}$ is a Parseval frame of $\J_\W(\omega_0)$
$$
\sum_{i = 1}^\kappa \sum_{g \in G} |\langle h_{i,g}(\omega),c\rangle_{\ell_2(\ZZdual)}|^2 = \|r_{g_0}c\|^2_{\ell_2(\ZZdual)} = \|c\|^2_{\ell_2(\ZZdual)} .
$$
By \cite[Theorem 4.1]{CP2010}, we then have that $\{\tr{k}\rot{g}\psi_i \, : \ i \in I_\kappa, k \in \ZZ, g \in G\}$ is a Parseval frame of $\W$ as wanted. This concludes the proof of $1.$

To prove the formula for the approximation error given in $2.$ recall that, by the Schmidt-Eckart-Young theorem for the data $\{a_{i,g}(\omega) \, , \ (i,g) \in I_m \times G\}$, with $\omega \in \Omega_0$ (see \cite[Theorem 4.1]{ACHM2007} for a version adapted to our situation)
$$
\mathcal D (\W; \data)(\omega) = \sum_{i = \kappa+1}^{m}\sum_{g \in G} \sigma_{i,g}(\omega)^2.
$$
Thus, 2. is a consequence of (\ref{Pro:minimize}) in Proposition \ref{prop:errors}.

To prove 3., we need to compute $\T[\psi_i] = h_{i,\id}$ explicitly. From the definition of $U(\omega)$, that is $U(\omega) = F(\omega)V(\omega)\Sigma^+(\omega)$, a direct computation shows that the $(i,g)$-th column of $U(\omega)$ is
\begin{equation}\label{eq:columns}
u_{i,g}(\omega) = \theta_{i,g}(\omega) \sum_{(j,q) \in I_m\times G} V_{j,q}^{i,g}(\omega) \, a_{j,q}(\omega) \in \ell_2(\ZZdual) \quad \textnormal{a.e. } \omega \in \Omega_0 .
\end{equation}
Using now (\ref{eq:extensionh}) we obtain
\begin{align*}
h_{i,\id}(\omega) & = \sum_{g \in G} r_{g^{-1}} u_{i,g}((g^{-1})^*\omega) \one{g^*\Omega_0}(\omega)\\
& = \sum_{g \in G} r_{g^{-1}} \theta_{i,g}((g^{-1})^*\omega) \sum_{(j,q) \in I_m\times G} V_{j,q}^{i,g} ((g^{-1})^*\omega)\, a_{j,q}((g^{-1})^*\omega) \one{g^*\Omega_0}(\omega)\\
& = \sum_{g \in G} \theta_{i,g}(\omega) \sum_{(j,q) \in I_m\times G} V_{j,g^{-1}q}^{i,g} (\omega)\, r_{g^{-1}} a_{j,q}((g^{-1})^*\omega) \one{g^*\Omega_0}(\omega)
\end{align*}
where the last identity is due to (\ref{eq:eigenvaluesinvariance}) and (\ref{eq:eigenvectorscovariance}).

By Lemma \ref{lem:intertwining} and the definition of $a_{i,g}$
$$
r_{g^{-1}} a_{j,q}((g^{-1})^*\omega) = r_{g^{-1}} \T[\rot{q}f_j]((g^{-1})^*\omega) = \T[\rot{g^{-1}q}f_j](\omega) .
$$
Therefore
\begin{align*}
h_{i,\id}(\omega) & = \sum_{g,q \in G} \sum_{j=1}^m \theta_{i,g}(\omega) V_{j,g^{-1}q}^{i,g} (\omega)\, a_{j,g^{-1}q}(\omega) \one{g^*\Omega_0}(\omega)\\
& \stackrel{g^{-1}q=g'}{=} \sum_{(j,g') \in I_m\times G} a_{j,g'}(\omega) \sum_{g \in G} \theta_{i,g}(\omega) V_{j,g'}^{i,g} (\omega) \one{g^*\Omega_0}(\omega).
\end{align*}

In terms of the coefficients (\ref{eq:coefficients}) we obtain
$$
\T[\psi_i](\omega) = \sum_{(j,g') \in I_m\times G} C_i^{j,g'}(\omega) \T[\rot{g'}f_j](\omega)
$$
which is (\ref{eq:generators}). This concludes the proof. \hfill \qed

\begin{rem}
Note that, by (\ref{eq:generators}) and Theorem \ref{Th:Helson}, we have that each $\psi_i$ belongs to $S_\Gamma(\data)$, so in particular $\W \subset S_\Gamma(\data)$.
\end{rem}

\noindent
{\bf{Acknowledgements.}}
We thank the anonymous referees for several comments and suggestions that have improved the presentation of this paper.

\end{document}